\documentclass{gtpart}
\usepackage{amsmath,amssymb,amsthm,stmaryrd}
\usepackage{mathrsfs}
\usepackage[all]{xy}
\usepackage[usenames,dvipsnames]{xcolor}
\usepackage{tikz}
\usetikzlibrary{shapes,positioning,fit}
\usepackage{url}
\usepackage{hyperref}
\usepackage{enumitem}
\usepackage{tensor}
\usepackage{graphicx}
\usepackage{mathtools}
\usetikzlibrary{arrows}
\DeclareMathAlphabet{\pazocal}{OMS}{zplm}{m}{n}

\doubletitle[Semistable models and power operations]{Semistable models for 
modular curves and power operations for Morava E-theories of height 2}
\author{Yifei Zhu}
\givenname{Yifei}
\surname{Zhu}
\address{Department of Mathematics\\Southern University of Science and 
         Technology\\Shenzhen\\Guangdong 518055\\P. R. China}
\email{zyf@umn.edu}
\urladdr{https://yifeizhu.github.io}


\newtheorem{thm}[equation]{Theorem}

\newtheorem{prop}[equation]{Proposition}
\newtheorem{lem}[equation]{Lemma}
\theoremstyle{definition}
\newtheorem{defn}[equation]{Definition}

\theoremstyle{remark}
\newtheorem{rmk}[equation]{Remark}
\newtheorem{ex}[equation]{Example}
\newtheorem{caut}[equation]{Caution}
\newtheorem{conv}[equation]{Convention}

\def\co{\colon\thinspace}
\newcommand{\mb}[1]{\mathbb{#1}}

\newcommand{\Spec}{{\rm Spec}}

\newcommand{\Spf}{{\rm Spf}}

\newcommand{\cF}{\overline {\mb F}}

\newcommand{\CC}{{\mathscr C}}

\newcommand{\CG}{{\mathscr G}}

\newcommand{\CM}{{\mathscr M}}
\newcommand{\CMB}{\overline{\CM}\!}
\newcommand{\CO}{{\mathscr O}}
\newcommand{\CP}{{\mathscr P}}

\newcommand{\eg}{{\em e.g.}}

\newcommand{\Sq}{{\rm Sq}}
\newcommand{\Frob}{{\rm Frob}}

\newcommand{\BC}{{\mb C}}

\newcommand{\BF}{{\mb F}}
\newcommand{\BG}{{\mb G}}

\newcommand{\BS}{{\mb S}}

\newcommand{\BZ}{{\mb Z}}

\newcommand{\HC}{\widehat{C~}\!}
\newcommand{\HCC}{\widehat{~\!\!\CC}\!}

\newcommand{\ea}{{\em et al.}}
\newcommand{\ib}{{\em ibid.}}
\newcommand{\ie}{{\em i.e.}}

\newcommand{\TE}{\widetilde{E\thinspace}\!}

\newcommand{\tA}{\widetilde{\A}}
\renewcommand{\th}{\widetilde{h}}
\newcommand{\tj}{\widetilde{j}}
\newcommand{\tu}{\widetilde{u}}
\newcommand{\tw}{\widetilde{w}}
\newcommand{\md}{\,~{\rm mod}~}

\newcommand{\can}{{\rm can}}

\newcommand{\Sub}{{\rm Sub}}
\newcommand{\A}{\alpha}
\newcommand{\B}{\beta}

\newcommand{\G}{\Gamma}

\newcommand{\K}{\kappa}

\newcommand{\mmu}{\mu\hskip-1.8mm\textcolor{white}{\spadesuit}\hskip-3mm\mu}
\renewcommand{\o}{\omega}

\newcommand{\T}{\tau}

\newcommand{\p}{\psi^{\,p}}

\newcommand{\ce}{\coloneqq}
\newcommand{\lb}{\llbracket}
\newcommand{\rb}{\rrbracket}
\newcommand{\lp}{(\!(}
\newcommand{\rp}{)\!)}

\newcommand{\wt}[1]{\textcolor{white}{#1} \!~}

\newcommand{\SL}{{\rm SL}}
\newcommand{\Tate}{{\rm Tate}}
\newcommand{\ch}[2]{{#1 \choose #2}}

\makeatletter
\DeclareRobustCommand\widecheck[1]{{\mathpalette\@widecheck{#1}}}
\def\@widecheck#1#2{%
    \setbox\z@\hbox{\m@th$#1#2$}%
    \setbox\tw@\hbox{\m@th$#1%
       \widehat{%
          \vrule\@width\z@\@height\ht\z@
          \vrule\@height\z@\@width\wd\z@}$}%
    \dp\tw@-\ht\z@
    \@tempdima\ht\z@ \advance\@tempdima2\ht\tw@ \divide\@tempdima\thr@@
    \setbox\tw@\hbox{%
       \raise\@tempdima\hbox{\scalebox{1}[-1]{\lower\@tempdima\box
\tw@}}}%
    {\ooalign{\box\tw@ \cr \box\z@}}}
\makeatother

\numberwithin{equation}{section}
\numberwithin{figure}{section}
\renewcommand{\theequation}{\thesection.\arabic{equation}}

\begin{document}

\begin{abstract}
 We construct an integral model for Lubin-Tate curves as moduli of finite 
 subgroups of formal deformations over complete Noetherian local rings.  They 
 are $p$-adic completions of the modular curves $X_0(p)$ at a mod-$p$ 
 supersingular point.  Our model is semistable in the sense that the only 
 singularities of its special fiber are normal crossings.  Given this model, 
 we obtain a uniform presentation for the Dyer-Lashof algebra of Morava 
 E-theories at height $2$ as local moduli of power operations in elliptic 
 cohomology.  
\end{abstract}

\maketitle

\section{Introduction}
\label{sec:intro}

\subsection{Overview}
\label{subsec:overview}

Understanding various sorts of moduli spaces is central to contemporary 
mathematics.  In algebraic geometry, (elliptic) modular curves parametrize 
elliptic curves equipped with level structures, which specify features 
attached to the group structure of an elliptic curve.  Each type of level 
structures corresponds to a specific subgroup of the modular group 
$\SL_2(\BZ)$, which acts on the upper half of the complex plane.  The 
complex-analytic model for a modular curve is the quotient of the upper-half 
plane by the action of such a subgroup, as a Riemann surface.  

Over $\BZ$, Deligne and Rapoport initiated the study of semistable models for 
the modular curves $X_0(N\,\!p)$ \cite[VI.6]{DR}.  The affine curve $xy = p$ 
over $W(\cF_p)$ appeared in their work as a local model for $X_0(p)$ near a 
mod-$p$ supersingular point.  Recently, Weinstein produced semistable models 
for Lubin-Tate curves (at height $2$) by passing to the infinite 
$p^\infty$-level, when they each have the structure of a perfectoid space 
\cite{Weinstein}.  Such a Lubin-Tate curve is the rigid space attached to 
the $p$-adic completion of a modular curve at one of its mod-$p$ 
supersingular points.  As the supersingular locus is the interesting part of 
the special fiber of a modular curve, Weinstein's work essentially provides 
semistable models for $X(N\,\!p^m)$.  His affine models include curves with 
equations $xy^{\,q}-x^{\,q}y = 1$ and $y^{\,q}+y = x^{\,q+1}$ over $\cF_q$, 
where $q$ is a power of $p \neq 2$.  

In this paper, with motivation from algebraic topology, we construct a new 
semistable model over $W(\cF_p)$ for the modular curve $X_0(p)$ near a 
mod-$p$ supersingular point.  The equation (see \eqref{me} below) for our 
integral affine model is more complicated than that of the Deligne-Rapoport 
model, while it reduces modulo $p$ to $x\,(y-x^{\,p}) = 0$.  The integral 
modular equation is essential to our application which produces an explicit 
presentation for the Dyer-Lashof algebra of a Morava E-theory at height 2, 
uniform with all $p$.  

The Dyer-Lashof algebra governs power operations on Morava E-theory as a 
generalized cohomology theory \cite{cong}.  As cohomology operations are 
natural transformations between functors, this algebra is a moduli space.  
Indeed, over the sphere spectrum $\BS$, Morava E-theories are topological 
realizations of Lubin-Tate curves of level $1$.  Their operations are thus 
parametrized by Lubin-Tate curves of higher levels.  

To be more precise, our results fit into the following framework.  
\[
 \begin{tikzpicture}
  \node (LT) at (0,3) {$\text{moduli of $\pazocal E\ell\ell_G$}$};
  \node (LT') at (0,2.6) {$\scriptstyle\text{to be understood as $G$ 
  varies}$};
  \node (MT) at (4.3,3) {$\text{moduli of $\pazocal E\ell\ell_{\G_1(N)}$}$};
  \node (RT) at (9.7,3) {$\text{\fbox{moduli of $\CC_N$}}$};
  \node (LB) at (0,0) {$\text{moduli of $\pazocal Quasi\,\pazocal 
        E\ell\ell_G$}$};
  \node (MB0) at (3.95,0) {$\textcolor{white}{\text{\fbox{moduli of $E$}}}$};
  \node (MB) at (4.3,0) {$\text{\fbox{moduli of $E$}}$};
  \node (MB1) at (4.3,-.5) {$\scriptstyle\text{Theorem \ref{thm:DL}}$};
  \node (MB') at (4.3,-3) {$\text{moduli of $L_{K(1)}E$}$};
  \node (MB'') at (4.3,-3.5) {$\scriptstyle\text{punctured formal 
  neighborhood of the}$};
  \node (MB''') at (4.3,-3.8) {$\scriptstyle\text{supersingular point as an 
  ordinary locus}$};
  \node (RB) at (9.7,0) {$\text{\fbox{moduli of $\BG$}}$};
  \node (RB1) at (9.7,-.5) {$\scriptstyle\text{Theorem \ref{thm:me}}$};
  \node (B) at (2.15,-1.5) {$\text{has power operations}$};
  \draw [->] (LT) -- node [above] {$\scriptstyle G\,=\,\G_1(N)$} (MT);
  \draw [->] (MT) -- node [above] {$\scriptstyle\text{derived version of}$} 
        (RT);
  \draw [->] (LT') -- node [left] {$\scriptstyle\text{at cusps}$} (LB);
  \draw [->] (MT) -- (MB);
  \draw [<->] (RT) -- node [left] {$\scriptstyle\text{at a supersingular 
        point}$} (RB);
  \draw [<->] (RT) -- node [right] {$\scriptstyle\text{Serre-Tate}$} (RB);
  \draw [<->] (MB) -- node [above] 
        {$\scriptstyle\text{Ando-Hopkins-Strickland}$} (RB);
  \draw [<->] (MB) -- node [below] {$\scriptstyle\text{Rezk}$} (RB);
  \draw [<->] (LB) to [out=270,in=180] node [left] 
        {$\scriptstyle\text{related (also to $K_\Tate$)}$} (MB');
  \draw [->] (MB1) -- node [right] {$\scriptstyle\text{$K(1)$-localization}$} 
        (MB');
  \draw [->,double] (B) -- node [right] {$\scriptstyle\text{~Huan}$} (LB);
  \draw [->,double] (B) -- (MB0);
  \draw [->,double] (B) -- (MB');
 \end{tikzpicture}
\]
In this diagram, the arrows between the boxed regions establish the 
aforementioned correspondences among modular curves, Lubin-Tate curves, and 
Dyler-Lashof algebras \cite{LST,AHS04,cong}.  Here, $E$ is a Morava E-theory 
spectrum of height $2$ at the prime $p$, $\BG$ is the formal group of $E$ as 
a Lubin-Tate universal deformation, and $\CC_N$ is a universal elliptic curve 
equipped with a level-$\G_1(N)$ structure whose formal group is isomorphic to 
$\BG$ (see Section \ref{sec:models} below).  Conjecturally, the moduli of $E$ 
should be the restriction of a moduli for a suitable equivariant elliptic 
cohomology theory $\pazocal E\ell\ell_G$ 
\cite{survey,global,Huan,RezkCorr13}.  

This paper represents an attempt to understand the above picture by working 
explicitly through the boxed regions.  To obtain the local model for 
$X_0(p)$, our strategies can be summarized in the following diagram 
(cf.~Figure \ref{fig} below), where we exploit the effectiveness of a modular 
form in terms of its calculable invariants, \ie, its weight, level, and a 
finite number of its first Fourier coefficients.  
\[
 \begin{tikzpicture}[baseline={([yshift=-10pt]current bounding box.north)}]
  \node (MT') at (3,2.75) {$\text{global moduli}$};
  \node (MT) at (3, 2.45) {$\scriptstyle\text{(modular forms)}$};
  \node (LB) at (0,0) {$\text{punctured formal neighborhood}$};
  \node (LB') at (0,-.5) {$\text{around the cusps}$};
  \node (LB'') at (0,-.9) {$\scriptstyle\text{(modular forms evaluated at 
        Tate curves)}$};
  \node (RB) at (6,0) {$\text{formal neighborhood}$};
  \node (RB') at (6,-.5) {$\text{of a supersingular point}$};
  \node (RB'') at (6,-.9) {$\scriptstyle\text{(local parameters / functions 
        on Lubin-Tate curves)}$};
  \draw [->] (MT) -- node [left] {$\scriptstyle\text{restriction / 
        completion}$} (LB);
  \draw [->] (MT) -- node [right] {$\scriptstyle\text{restriction / 
        completion}$} (RB);
 \end{tikzpicture}
\]

We now explain our main results in more detail and state the theorems in 
Sections \ref{subsec:modag} and \ref{subsec:modat}, respectively, concerning 
aspects of algebraic geometry and algebraic topology.

\subsection{Moduli of elliptic curves and of formal groups, and Theorem 
\ref{thm:me}}
\label{subsec:modag}

Known to Kronecker, the congruence 
\begin{equation}
 \label{Kronecker}
 (\,j-\tj^{\,p})(\,\tj-j^{\,p})\equiv0\md p 
\end{equation}
gives an equation for the modular curve $X_0(p)$ which represents (in a 
relative sense, cf.~Section \ref{subsec:levelp} below) the moduli problem 
$[\G_0(p)]$ for elliptic curves over a perfect field of characteristic $p$.  
This moduli problem associates to such an elliptic curve its finite flat 
subgroup schemes of rank $p$.  A choice of such a subgroup scheme is 
equivalent to a choice of an isogeny from the elliptic curve with a 
prescribed kernel.  The $j$-invariants of the source and target curves along 
this isogeny are parametrized by $j$ and $\tj$.  

More precisely, this Kronecker congruence provides a {\em local} description 
for $[\G_0(p)]$ at a supersingular point.  For large primes $p$, the mod-$p$ 
supersingular locus may consist of more than one closed point.  In this case, 
$X_0(p)$ does not have an equation in the simple form above.  Only its 
completion at a single supersingular point does.  

There are polynomials which describe $X_0(p)$ as a curve over $\Spec(\BZ)$.  
In the Modular Polynomial Databases of the computational algebra software 
\href{http://magma.maths.usyd.edu.au/magma/handbook/modular_curves}{Magma}, 
{\em classical modular polynomials} lift and globalize the Kronecker 
congruence, while {\em canonical modular polynomials}, in a different 
pair of parameters, appear simpler.  Here is a sample of the latter, with the 
first three modular curves of genus $0$ and the last one of genus $1$ 
(cf.~Figure \ref{fig} below for $p = 11$).  
\[
 \begin{split}
   X_0(2)\qquad & ~x^3+48x^2+(768-j)x+2^{12} \\
         \equiv & ~x(x^2-j)\md2 \\
   X_0(3)\qquad & ~x^4+36 x^3+270 x^2+(756-j) x+3^6 \\
         \equiv & ~x(x^3-j)\md3 \\
   X_0(5)\qquad & ~x^6+30x^5+315x^4+1300x^3+1575x^2+(750-j)x+5^3 \\
         \equiv & ~x(x^5-j)\md5 \\
  X_0(11)\qquad & ~x^{12}-5940x^{11}+14701434x^{10}+(-139755\,j-19264518900) 
                  x^9 \\
                & +(723797800\,j+13849401061815)x^8+(67496\,j^2-1327909897380 
                  \,j \\
                & -4875351166521000)x^7+(2291468355\,j^2+1036871615940600\,j 
                  \\
                & +400050977713074380)x^6+(-5346\,j^3+4231762569540\,j^2 \\
                & -310557763459301490\,j+122471154456433615800)x^5 \\
                & +(161201040\,j^3+755793774757450\,j^2+17309546645642506200 
                  \,j \\
                & +6513391734069824031615)x^4+(132\,j^4-49836805205\,j^3 \\
                & +6941543075967060\,j^2-64815179429761398660\,j \\
                & +104264884483130180036700)x^3+(468754\,j^4+51801406800\,j^3 
                  \\
                & +214437541826475\,j^2+77380735840203400\,j \\
                & +804140494949359194)x^2+(-\,j^5+3732\,j^4-4586706\,j^3 \\
                & +2059075976\,j^2-253478654715\,j+2067305393340)x+11^6 \\
         \equiv & ~x\big(x^{11}-j^2(\,j-1)^3\big)\md11 
 \end{split}
\]
In these canonical modular polynomials for $X_0(p)$, the variable 
\begin{equation}
\label{etaquot}
 x = x(z) \ce p^s\!\left[\frac{\eta(pz)}{\eta(z)}\right]^{\!2s} 
\end{equation}
where $\eta$ is the Dedekind $\eta$-function and $s = 12/\!\gcd(p-1,12)$ ($s$ 
equals the exponent in the constant term of each polynomial).  The 
Atkin-Lehner involution (see Section \ref{subsec:AL} below) sends $x$ to 
$p^s/x$.  Computing these polynomials can be difficult.  As Milne warns in 
\cite[Section 6]{Milne}, ``one gets nowhere with brute force methods in this 
subject.''  Fortunately, for our purpose, we need only a suitable {\em local} 
(but still integral) equation for $X_0(p)$ completed at a single mod-$p$ 
supersingular point, which we shall present below as a variant of the above 
polynomials in $j$ and $x$.  

As mentioned earlier, this completion of $X_0(p)$ is a Lubin-Tate curve, 
which is a moduli space for formal groups.  Indeed, there is a connection 
between the moduli of formal groups and the moduli of elliptic curves.  This 
is the Serre-Tate theorem, which states that $p$-adically, the deformation 
theory of an elliptic curve is equivalent to the deformation theory of its 
$p$-divisible group \cite[Section 6]{LST}.  In particular, the $p$-divisible 
group of a supersingular elliptic curve is formal.  Thus the local 
information provided by the Kronecker congruence (and its integral lifts) 
becomes useful for understanding deformations of formal groups of height 2.  

Lubin and Tate developed the deformation theory for one-dimensional formal 
groups of finite height \cite[esp.~Theorem 3.1]{LT}.  More recently, with 
motivation from algebraic topology, Strickland studied the classification of 
finite subgroups of Lubin-Tate universal deformations.  In particular, he 
proved a representability theorem for this moduli of deformations 
\cite[Theorem 42]{Str97}.  The representing objects, each a Gorenstein affine 
formal scheme, are precisely what we will call {\em Lubin-Tate curves of 
level $\G_0(p^m)$}.  

Our first main result gives an explicit model for Lubin-Tate curves of level 
$\G_0(p)$ over the Witt ring $W(\cF_p)$, whose special fiber is {\em 
semistable} in the sense that its only singularities are normal crossings.  
Equivalently, this describes the complete local ring of the modular curve 
$X_0(p)$ at a mod-$p$ supersingular point in terms of generators and 
relations.  
\begin{thm}
 \label{thm:me}
 Let $\BG_0$ be a formal group over $\cF_p$ of height $2$ and let $\BG$ be 
 its universal deformation over the Lubin-Tate ring.  For each $m\geq0$, 
 denote by $A_m$ the ring which classifies degree-$p^m$ subgroups of the 
 formal group $\BG$.  It is the ring of functions on the Lubin-Tate curve of 
 level $\G_0(p^m)$.  In particular, write $A_0 \cong W(\cF_p)\lb h\rb$ for 
 the Lubin-Tate ring.  

 Then, when $m = 1$, the ring $A_1 \cong 
 W(\cF_p)\lb h,\A\rb\,\big/\big(w(h,\A)\big)$ is determined by the polynomial 
 \begin{equation}
  \label{me}
  w(h,\A) = (\A-p)\big(\A+(-1)^{\,p}\big)^p-\big(h-p^2+(-1)^{\,p}\big)\A 
 \end{equation}
 which reduces to $\A(\A^{\,p}-h)$ modulo $p$.  
\end{thm}

\begin{rmk}
 The rings $A_m$ are denoted by $\CO_{\Sub_m(\BG)}$ in \cite{Str97}.  
 As a result of a different choice of parameters, the last congruence above 
 is not in the form \eqref{Kronecker} of Kronecker's.  The letter $h$ stands 
 for ``Hasse'' in the Hasse invariant.  See Section \ref{subsec:parameters} 
 below for details about parameters.  See also Remark \ref{rmk:choices} for 
 dependence of \eqref{me} on the choices involved for different primes $p$.  
\end{rmk}

\begin{rmk}
 Let $\tA$ denote the image of $\A$ under the Atkin-Lehner involution (see 
 Section \ref{subsec:AL} below).  We will show in \eqref{AA} that $\A\cdot\tA 
 = (-1)^{\,p-1}p$ (cf.~$xy = p$ in the Deligne-Rapoport model in Section 
 \ref{subsec:overview}).  Note that the product equals the constant term of 
 \eqref{me} as a polynomial in $\A$ of degree $p+1$.  Thus factoring out $\A$ 
 from the modular equation $w(h,\A) = 0$, we obtain a congruence 
 \[
  h\equiv\A^{\,p}+\tA\md p 
 \]
 This is a manifest of the Eichler-Shimura relation $T_p\equiv F+V\md p$ 
 between the Hecke, Frobenius, and Verschiebung operators, which reinterprets 
 the moduli problem $[\G_0(p)]$ in characteristic $p$ (see below in this 
 remark and the discussion of dual isogenies in Section \ref{subsec:cusps}).  
 
 The polynomial $w(h,\A)$ can be viewed as a local variant of a canonical 
 modular polynomial, whose parameters are the $j$-invariant and an 
 eta-quotient \eqref{etaquot}.  In fact, the key step \eqref{key} in our 
 proof of Theorem \ref{thm:me} below was inspired by \cite[Example 2.4, 
 esp.~(2.4)]{Choi}.  
 
 Related to Choi's work, compare the sequences 
 $\{\,j_n^{\,(p)}(z)\}_{n=1}^\infty$ where $p\in\{2,3,5,7,13\}$\footnote{This 
 constraint on $p$ is necessary for univalence of (global) modular functions 
 on zero-genus congruence subgroups.  Cf.~Lemma \ref{lem:expansion} below.  } 
 in \cite{Ahlgren} and $\{\,j_m(z)\}_{m=0}^\infty$ in \cite{BKO} of Hecke 
 translates of Hauptmoduln.  They are analogous to the relation 
 $h\equiv T_p\,\A\md p$ from above.  
\end{rmk}

\subsection{Moduli of Morava E-theory spectra and Theorem \ref{thm:DL}}
\label{subsec:modat}

The Adem relations 
\begin{equation}
 \label{Adem}
 \Sq^i\,\Sq^{\,j} = \sum_{k=0}^{\left\lfloor\frac{i}{2}\right\rfloor} 
 \ch{j-k-1}{i-2k}\Sq^{i+ j-k}\,\Sq^k \hskip2cm 0<i<2\,j 
\end{equation}
describe the rule of multiplication (composition of cohomology operations) 
for the Steenrod squares $\Sq^i$.  These are power operations in ordinary 
cohomology with $\BF_2$-coefficients.  In general, for ordinary cohomology 
with $\BF_p$-coefficients, the collection of its power operations has the 
structure of a graded Hopf algebra over $\BF_p$, called the {\em mod-$p$ 
Steenrod algebra} (cf.~Dyer-Lashof operations in, \eg, \cite[Theorem 
III.1.1]{H_infty} and see \eqref{total} below).  

Quillen's work connects complex cobordism and the theory of one-dimensional 
formal groups \cite{Quillen}.  This leads to a {\em height filtration} for 
the stable homotopy category, which has turned out to be a highly effective 
principle since the 1980s for organizing large-scale periodic phenomena in 
the stable homotopy groups of spheres \cite{DHS,HS,orange}.  The height of a 
formal group indicates the filtration level of its corresponding generalized 
cohomology theory.  Ordinary cohomology theories with $\BF_p$-coefficients 
fit into this framework of {\em chromatic homotopy theory}, as theories 
concentrated at height $\infty$.  

The power operation algebras for cohomology theories at other chromatic 
levels have been studied as well.  In particular, central to the chromatic 
viewpoint is a family of Morava E-theories, one for each finite height $n$ at 
a particular prime $p$.  More precisely, given any formal group $\BG_0$ of 
height $n$ over a perfect field of characteristic $p$, there is a Morava 
E-theory associated to the Lubin-Tate universal deformation of $\BG_0$.  Via 
Bousfield localizations, these Morava E-theories determine the chromatic 
filtration of the stable homotopy category.  

There is a connection between (stable) power operations in a Morava E-theory 
$E$ and deformations of powers of Frobenius on its corresponding formal 
group $\BG_0$.  This is Rezk's theorem, built on the work of Ando, Hopkins, 
and Strickland \cite{Ando95, Str97, Str98, AHS04}.  It gives an equivalence 
of categories between (i) graded commutative algebras over a Dyer-Lashof 
algebra for $E$ and (ii) quasicoherent sheaves of graded commutative algebras 
over the moduli problem of deformations of $\BG_0$ and Frobenius isogenies 
\cite[Theorem B]{cong}.  Here, the Dyer-Lashof algebra is a collection of 
power operations that governs all homotopy operations on commutative 
$E$-algebra spectra \cite[Theorem A]{cong}.  

At height 2, information from the moduli of elliptic curves allows a concrete 
understanding of the power operation structure on Morava E-theories.  Rezk 
computed the first example of a presentation for an E-theory Dyer-Lashof 
algebra, in terms of explicit generators and quadratic relations analogous to 
the Adem relations \eqref{Adem} \cite{h2p2}.  Moreover, he gave a uniform 
presentation, which applies to E-theories at all primes $p$, for the mod-$p$ 
reduction of their Dyer-Lashof algebras \cite[4.8]{mc1}.  Underlying this 
presentation is the Kronecker congruence \eqref{Kronecker} \cite[Proposition 
3.15]{mc1}.  

Our second main result provides an ``integral lift'' of Rezk's mod-$p$ 
presentation, with a different set of generators, in the same sense that 
Theorem \ref{thm:me} above lifts the Kronecker congruence.  We begin with the 
following.  

\begin{thm}
 \label{thm:po}
 Let $E$ be a Morava E-theory spectrum of height $2$ at the prime $p$.  There 
 is an additive total power operation 
 \[
  \begin{split}
          \p\co E^0 \to & ~E^0(B\Sigma_p)/I \\
   W(\cF_p)\lb h\rb \to & ~W(\cF_p)\lb h,\A\rb\,\big/\big(w(h,\A)\big) 
  \end{split}
 \]
 where $I$ is a transfer ideal.  
 
 \begin{enumerate}[label={\em(\roman*)}]
  \item The polynomial 
  \[
   w(h,\A) = w_{p+1}\A^{\,p+1}+\cdots+w_1\A+w_0 \hskip2cm w_i \in E^0 
  \]
  can be given as \eqref{me} from Theorem \ref{thm:me}.  In particular, 
  $w_{p+1} = 1$, $w_1 = -h$, $w_0 = (-1)^{\,p+1}p$, and the remaining 
  coefficients 
  \[
   w_i = (-1)^{\,p\,(\,p-i+1)}\! 
   \left[\ch{p}{i-1}\!+(-1)^{\,p+1}\,\!p\,\!\ch{\,p\,}{i}\right] 
  \]
  
  \item The image $\p(h) = \sum_{i=0}^pQ_i(h)\,\A^i$ is then given by 
  \[
   \p(h) = \A+\sum_{i=0}^p\A^i\sum_{\T=1}^pw_{\T+1}\,d_{i,\T} 
  \]
  where 
  \[
   d_{i,\T} = \sum_{n=0}^{\T-1}(-1)^{\T-n}\,w_0^n\sum_{\stackrel{\scriptstyle 
   m_1+\cdots+m_{\T-n}=\T+i}{\stackrel{\scriptstyle 1\,\leq\,m_s\,\leq\,p+1} 
   {m_{\T-n}\,\geq\,i+1}}}\!w_{m_1}\cdots w_{m_{\T-n}} 
  \]
  In particular, $Q_0(h)\equiv h^{\,p}\md p$.  
 \end{enumerate}
\end{thm}

This leads to the second main theorem of the paper.  

\begin{thm}
 \label{thm:DL}
 Continue with the notation in Theorem \ref{thm:po}.  Let $\G$ be the 
 Dyer-Lashof algebra for $E$, which is the ring of additive power operations 
 on $K(2)$-local commutative $E$-algebras.  
 
 Then $\G$ admits a presentation as the associative ring generated over $E^0 
 \cong W(\cF_p)\lb h\rb$ by elements $Q_i$, $0\leq i\leq p$, subject to the 
 following set of relations.  
 
 \begin{enumerate}[label={\em(\roman*)}]
  \item Adem relations 
  \[
   Q_kQ_0 = -\sum_{j=1}^{p-k}w_0^{\,j}\,Q_{k+j}Q_j\,-\sum_{j=1}^p 
   \sum_{i=0}^{j-1}w_0^i\,d_{k,\,j-i}\,Q_iQ_j \hskip2.3cm \text{for 
   $1\leq k\leq p$} 
  \]
  where the first summation is vacuous if $k = p$.  
  
  \item Commutation relations 
  \[
   \begin{split}
    Q_i\,c = & ~(Fc)\,Q_i~~~\text{for $\,c \in W(\cF_p)$ and all $i$, with 
               $F$ the Frobenius automorphism} \\
    Q_0\,h = & ~e_0+(-1)^{\,p+1}r\sum_{m=0}^{p-1}s^me_{p+m+1}+(-1)^{\,p} 
               \Bigg(e_p+r\,e_{2p}+\sum_{m=1}^ps^me_{p+m}\Bigg) \\
             & +\sum_{j=1}^{p-1}\,(-1)^{\,p\,j}\Bigg[e_j+r\,s^{\,p-j}e_{2p}+ 
               r\!\sum_{m=0}^{p-j-1}s^m\big(e_{p+j+m}+(-1)^{\,p+1}e_{p+j+m+1} 
               \big)\Bigg] \\
    Q_k\,h = & ~(-1)^{\,p\,(p-k)}\ch{p}{k}\Bigg(e_p+r\,e_{2p}+\sum_{m=1}^p 
               s^me_{p+m}\Bigg)+\sum_{j=k}^{p-1}\,(-1)^{\,p\,(\,j-k)} 
               \ch{\,j\,}{k}\Bigg[ e_j \\
             & +r\,s^{\,p-j}e_{2p}+r\!\sum_{m=0}^{p-j-1}s^m\big(e_{p+j+m}+ 
               (-1)^{\,p+1}e_{p+j+m+1}\big)\Bigg]~~~\text{for \,\! $0<k<p$} 
               \\
    Q_p\,h = & ~e_p+r\,e_{2p}+\sum_{m=1}^ps^me_{p+m} 
   \end{split}
  \]
  where $r = h-p^2+(-1)^{\,p}$, $s = p+(-1)^{\,p}$, and 
  \[
   \begin{split}
    e_n = & ~\sum_{m=n}^{p+1}(-1)^{(p+1)(m-n)}\ch{m}{n}Q_{m-1} \\
          & +\sum_{m=n}^{2p}(-1)^{(p+1)(m-n)}\ch{m}{n}\!\!\sum_{\stackrel
            {\scriptstyle i+j=m}{0\,\leq\,i,\,j\,\leq\,p}}\sum_{\T=1}^p 
            w_{\T+1}\,d_{i,\T}\,Q_j 
   \end{split}
  \]
  the first summation for $e_n$ being vacuous if $p+2\leq n\leq 2\,p$, and 
  being vacuous in its term $m = 0$ if $n = 0$.  
 \end{enumerate}
\end{thm}

The Dyer-Lashof algebra $\G$ has the structure of a twisted bialgebra over 
$E^0$.  The ``twists'' are described by the commutation relations above.  The 
product structure satisfies the Adem relations.  Certain Cartan formulas give 
rise to the coproduct structure as follows.  

\begin{prop}
 \label{prop:Cartan}
 Let $A$ be any $K(2)$-local commutative $E$-algebra.  There are additive 
 individual power operations $Q_k\co \pi_0(A) \to \pi_0(A)$, $0\leq k\leq p$, 
 which satisfy the following Cartan formulas as well as the Adem and 
 commutation relations from Theorem \ref{thm:DL}.  
 
 For each $0\leq k\leq p$, $Q_k(xy)$ equals the expression on the right-hand 
 side of the commutation relation for $Q_k\,h$, where $r = h-p^2+(-1)^{\,p}$ 
 and $s = p+(-1)^{\,p}$ as above, and 
 \[
  e_n = \sum_{m=n}^{2p}(-1)^{(p+1)(m-n)}\ch{m}{n}\!\!\sum_{\stackrel
  {\scriptstyle i+j=m}{0\,\leq\,i,\,j\,\leq\,p}}\!\!Q_i(x)\,Q_j(y) 
 \]
\end{prop}
\begin{proof}
 In Section \ref{subsec:comm} below proving Theorem \ref{thm:DL}, we will 
 present a proof for the commutation relations, in such a way that the same 
 formal procedure applies to give the stated Cartan formulas.  
\end{proof}

\begin{rmk}
 \label{rmk:recover}
 Theorem \ref{thm:DL} and Proposition \ref{prop:Cartan} recover corresponding 
 earlier results in \cite{h2p2,p3,ho} for $p = 2$, $3$, and $5$ 
 respectively.  
 
 To be more precise, for $p = 3$, the presentations do not 
 coincide but are equivalent.  Cf.~Example \ref{ex:p3N4} and Definition 
 \ref{def:pm} below after we discuss models for the total power operation 
 $\p$.  The equivalency will be addressed in Remark \ref{rmk:choices}.  
 Theorem \ref{thm:po} was stated with respect to a particular basis for the 
 target ring $E^0(B\Sigma_p)/I$ of $\p$ as a free module over $E^0$ of rank 
 $p+1$.  Theorem \ref{thm:DL} was stated with respect to a particular basis 
 for the Dyer-Lashof algebra $\G$ as an associative ring over $E^0$ on 
 $(\,p+1)$ generators.  
 
 For $p = 5$, in \cite[Example 6.1]{ho}, we did not completely determine the 
 Dyer-Lashof algebra due to constraints with our earlier methods (cf.~[\ib, 
 Example 3.5]).  The issue was that, after a quadratic extension of $\BF_5$, 
 the mod-$5$ Hasse invariant factors into a pair of Galois conjugates (see 
 Example \ref{ex:p5N4} below).  We calculated with global modular forms 
 instead of local functions on Lubin-Tate curves.  Thus we were unable to 
 determine the image under $\p$ of $h \in E^0$ (denoted \ib~by $u_1$, with 
 $h$ standing for the global Hasse invariant) which corresponds to a local 
 deformation parameter at one of the two supersingular points.  Nevertheless, 
 in fact, the current and earlier presentations for the Dyer-Lashof algebra 
 in this case coincide formally.  This is a manifest of the functoriality, 
 with respect to base field extension, of the Dyer-Lashof algebra as a moduli 
 space.  
\end{rmk}

\begin{caut}
 In this paper, we use the word ``model'' for two distinct but closely 
 related objects.  One is an algebraic curve with an explicit defining 
 equation as customary in algebraic geometry.  The other is a set of data 
 involving formal groups and elliptic curves which is designed to facilitate 
 explicit calculations in algebraic topology (Definitions \ref{def:cm}, 
 \ref{def:im}, and \ref{def:pm}).  
\end{caut}

\begin{rmk}
 Let $E$ be a Morava E-theory of height $n$.  Recent work of Behrens and Rezk 
 on spectral algebra models for unstable $v_n$-periodic homotopy theory has 
 identified (i) the completed $E$-homology of the $n$'th Bousfield-Kuhn 
 functor applied to an odd-dimensional sphere with (ii) the $E$-cohomology of 
 the $K(n)$-localized Andr\'e-Quillen homology of the spectrum of cochains on 
 the odd-dimensional sphere valued in the $K(n)$-local sphere spectrum 
 \cite[Theorem 8.1]{BKTAQ}.  The former (i) computes unstable $v_n$-periodic 
 homotopy groups of spheres via a homotopy fixed point spectral sequence of 
 Devinatz and Hopkins.  By calculations of Rezk, (ii) can be identified at 
 heights $n = 1$ and $n = 2$ via another spectral sequence, whose $E_2$-page 
 consists of Ext-groups of certain rank-one modules over the Dyer-Lashof 
 algebra of $E$ (see \cite[Example 2.13]{h2}).  
 
 We have applied Theorems \ref{thm:me} and \ref{thm:po} in this context of 
 {\em unstable} chromatic homotopy theory to make these Ext-groups more 
 explicit in \cite{bkos}, with subsequent work partly joint with Guozhen Wang 
 towards general patterns in higher chromatic levels inspired by 
 \cite{mc1,Weinstein}.  
\end{rmk}

\subsection{Outline for the rest of the paper}

In Section \ref{sec:models}, we collect, streamline, and provide an in-depth 
analysis of necessary preliminary materials from \cite[Sections 2.1, 2.2, 
3.1, and 3.3]{ho}.  The presentation here is self-contained.  

After a brief discussion of the homotopy-theoretic setup, we present in 
Sections \ref{subsec:level1} and \ref{subsec:levelp} {\em models} for Morava 
E-theories of height 2 and their power operations, which are built from data 
of formal groups and elliptic curves.  Section \ref{subsec:AL} is an 
interlude on the Atkin-Lehner involution, which becomes important later.  We 
then introduce parameters for the local formal moduli in Section 
\ref{subsec:parameters} and give a detailed analysis of their behavior at the 
cusps of the global moduli, as summarized in Lemma \ref{lem:cusps} and 
Definition \ref{def:pm}.  

Sections \ref{sec:pfag} and \ref{sec:pfat} are devoted to proving the main 
theorems stated in Sections \ref{subsec:modag} and \ref{subsec:modat} above, 
respectively, for the moduli spaces in algebraic geometry and in algebraic 
topology.  The key ingredients for the proof in Section \ref{sec:pfag} are 
Lemmas \ref{lem:cusps} and \ref{lem:expansion}.

\subsection{Acknowledgments}

I thank Mark Behrens, Elden Elmanto, Paul Goerss, Mike Hopkins, Zhen Huan, 
Tyler Lawson, Haynes Miller, Charles Rezk, Joel Specter, Guozhen Wang, and 
Zhouli Xu for helpful discussions and encouragements while the ideas and 
results presented in this paper evolve over the years.  

I thank my friends Tzu-Yu Liu and Meng Yu for their continued support, 
especially during me sketching a preliminary version of Section 
\ref{sec:pfag} in their home in California.  

This work was partially supported by National Natural Science Foundation of 
China grant 11701263.

\section{Modeling Morava E-theories of height $2$ via moduli spaces of 
elliptic curves}
\label{sec:models}

Given a formal group $\BG_0$ over $\cF_p$ of height $2$, let $E$ be the 
Morava E-theory spectrum associated to $\BG_0/\cF_p$ by the 
Goerss-Hopkins-Miller theorem \cite{GH}.  We have 
\[
 \pi_*(E) = W(\cF_p)\lb\mu_1\rb[\mu^{\pm1}] 
\]
with $|\mu_1| = 0$ and $|\mu| = 2$.  

Let $x_{_E} \in \TE^2(\BC P^\infty)$ be a class which extends to a complex 
orientation of $E$, so that $E^*(\BC P^\infty) \cong E^*\lb x_{_E}\rb$ and 
that the formal scheme $\Spf\big(E^0(\BC P^\infty)\big)$ has a group 
structure \cite[Section 1]{coctalos}.  In particular, $x_{_E}\cdot\mu$ is a 
coordinate on this formal group, \ie, a uniformizer for its ring of functions 
\cite[Definition 1.4 and Remark 1.7]{Ando00}.  

The degree-$0$ coefficient ring $E^0 = W(\cF_p)\lb\mu_1\rb$ is the Lubin-Tate 
ring which classifies formal deformations of $\BG_0$ \cite[Theorem 3.1]{LT}.  
The formal group $\Spf\big(E^0(\BC P^\infty)\big)$ is the universal 
deformation $\BG$ of $\BG_0$ over $E^0$.  

As an $E_\infty$-ring spectrum, $E$ affords power operations constructed from 
the extended power functors $D_m(-) \ce 
\big(-\big)^{\wedge_E\,m}_{h\Sigma_m}$ on $E$-modules, for each integer 
$m\geq0$, of taking the $m$-fold smash product over $E$ modulo the action of 
the symmetric group $\Sigma_m$ up to homotopy.  

In particular, we have the (additive) {\em total} $p$-power operation 
\begin{equation}
 \label{total}
 \p\co E^0 \to E^0(B\Sigma_p)/I 
\end{equation}
where $I = \bigoplus_{0<i<p}{\rm im}\big(E^0(B\Sigma_i\times B\Sigma_{p-i}) 
\to E^0(B\Sigma_p)\big)$ is the ideal generated by images of transfers.  It 
gives rise to {\em individual} power operations $Q_i\co \pi_0(A) \to 
\pi_0(A)$, $0\leq i\leq p$, for any $K(2)$-local commutative $E$-algebra $A$ 
(see \cite[Definition I.4.2]{H_infty} and \cite[Definition 3.5]{p3}).

\subsection{Models for an E-theory and Lubin-Tate curves of level $1$}
\label{subsec:level1}

Given such an E-theory above, to carry out explicit calculations for its 
power operations, we work with elliptic curves as models (cf.~\cite[Section 
2]{ho}).  

First, let $C_0$ be a supersingular elliptic curve over $\cF_p$.  Its formal 
completion $\HC_{\!0}$ at the identity section is a formal group of height 
$2$ over $\cF_p$.  By \cite[Th\'eor\`eme IV]{Lazard}, $\HC_{\!0}$ is 
isomorphic to $\BG_0$.  

Next, to associate an elliptic curve to the universal formal deformation 
$\BG$ of $\BG_0$, we apply the Serre-Tate theorem \cite{LST} 
(cf.~\cite[Theorem 2.9.1]{KM}) and construct a universal deformation of the 
elliptic curve $C_0$.  For representability, we equip $C_0$ with a 
level-$\G_1(N)$ structure, $N\geq3$ and $p\nmid N$.  

More precisely, consider the representable moduli problem $\CP_N$ of 
isomorphism classes of smooth elliptic curves over $\BZ[1/N]$ with a choice 
of a point of exact order $N$ and a nonvanishing $1$-form.  Let $\CM_N$ be 
its representing scheme,\footnote{See \cite[Proposition 3.2]{tmf3} and 
\cite[Examples 2.1 and 2.2]{ho} for explicit presentations in the cases when 
$N = 3$, $4$, and $5$.  } which is of relative dimension $1$ over 
$\BZ[1/N]$.  Let $\CC_N$ be the universal elliptic curve over this modular 
curve $\CM_N$.  In general, for $p$ and $N$ large, the supersingular locus of 
$\CM_N$ at $p$ consists of more than one closed point, and $C_0$ is the 
fiber of $\CC_N$ over one of them.  By the Serre-Tate theorem, the formal 
completion $\HCC_N$ of $\CC_N$ at the identity section is isomorphic to the 
universal formal deformation $\BG$ of $\BG_0 \cong \HC_{\!0}$.  

\begin{rmk}
 \label{rmk:indepc}
  We see that up to isomorphism, the E-theory $E$ associated to $\BG_0/\cF_p$ 
  does not depend on the choice of $C_0$ and $\CC_N$ that model its formal 
  group $\BG$.  
\end{rmk}

\begin{defn}{(\cite[Definition 2.8]{ho})}
 \label{def:cm}
 Let $E$ be a Morava E-theory of height $2$ at the prime $p$.  With notation 
 as above, a {\em $\CP_N$-model for $E$} is the following set of data.  
 \begin{enumerate}[label=Mod.\,\arabic*, leftmargin=1.8cm]
  \item \label{m1} a supersingular elliptic curve $C_0$ over $\cF_p$ 
  
  \item \label{m2} a universal deformation $\CC_N$ of $C_0$ over $\CM_N$ 
  
  \item \label{m3} a coordinate $u$ on the formal group $\HCC_N$ 
  
  \item \label{m4} an isomorphism between $\Spf(E^0)$ and the formal 
  completion of $\CM_N$ at the supersingular point corresponding to $C_0$ 
  
  \item \label{m5} an isomorphism between $\Spf\big(E^0(\BC P^\infty)\big)$ 
  and $\HCC_N$ as formal groups over $E^0$ which sends $x_{_E}\cdot\mu$ to 
  $u$ 
 \end{enumerate}
\end{defn}

As discussed above, the existence of the isomorphisms in \ref{m4} and 
\ref{m5} follows from the Lubin-Tate theorem combined with the Serre-Tate 
theorem.  

The formal schemes in \ref{m4} are each of relative dimension $1$ over the 
formal completion $\Spf\big(W(\cF_p)\big)$ of $\Spec(\BZ)$ at $p$.  We call 
them {\em Lubin-Tate curves of level $1$} (the level-$\G_1(N)$ structure is 
auxiliary for our purpose).

\subsection{Models for power operations on $E$ and Lubin-Tate curves of level 
$\Gamma_0(p)$}
\label{subsec:levelp}

By work of Ando, Hopkins, Rezk, and Strickland, power operations on $E$ 
correspond to finite flat subgroup schemes of $\BG$, or equivalently, to 
isogenies from $\BG$ \cite[Theorem B]{cong}.  Thus, again via the Serre-Tate 
theorem, we model the latter by isogenies between elliptic curves in order to 
obtain explicit formulas for the power operations.  

More precisely, as the base ring $E^0$ is $p$-local, we need only work with 
the moduli problems $[\G_0(p^r)]$ of isomorphism classes of elliptic curves 
with a choice of degree-$p^r$ subgroup scheme.  It suffices to analyze 
$[\G_0(p)]$ and its Atkin-Lehner involution.\footnote{See \cite[11.3.1]{KM} 
with more details below in Section \ref{subsec:AL}.  }  Indeed, the ring of 
(additive) power operations on any Morava E-theory is {\em quadratic} 
\cite[Main Theorem and Proposition 4.10]{Koszul}: with ring multiplication 
given by composition, the $p^r$-power operations are generated by the 
$p$-power ones, subject to quadratic relations which describe products of two 
generators.  

The moduli problem $[\G_0(p)]$ is {\em relatively} representable over the 
moduli stack of elliptic curves \cite[4.2 and 5.1.1]{KM}.  In particular, the 
simultaneous moduli problem $\CP_N\times[\G_0(p)]$ is representable by a 
scheme $\CM_{N,\,p}$ which is finite flat over $\CM_N$ of degree $p+1$.  

\begin{rmk}
 \label{rmk:indepi}
 The scheme $\CM_{N,\,p}$ is of relative dimension $1$ over $\BZ[1/N]$.  It 
 is commonly referred to as {\em the} modular curve of level $\G_0(p)$, whose 
 compactification is denoted by $X_0(p)$ in the literature.  Up to base 
 change, this modular curve is independent of the rigidification by $\CP_N$ 
 (for representability), or by any other types of level-$N$ structure 
 (cf.~\cite[Chapter 1, esp.~1.13]{padicprop} where full level-$\G(N)$ 
 structures are used).  Also cf.~\cite[13.4.7]{KM} where $\CP$ rigidifies 
 $[\G_0(p)]$.  In particular, varying $N$ will not change the local equation 
 for $\CM_{N,\,p}$ as a scheme over $\CM_N$.  
\end{rmk}

Now, we construct a universal degree-$p$ isogeny over $\CM_{N,\,p}$ as 
follows (cf.~\cite[Construction 3.1]{ho}).  Over $\CM_{N,\,p}$, let 
$\CG_N^{(p)}$ be the universal example of a degree-$p$ subgroup scheme of 
$\CC_N$ and write $\CC_N^{(p)} \ce \CC_N/\CG_N^{(p)}$ for the quotient 
elliptic curve.  Define $\Psi_N^{(p)}\co \CC_N \to \CC_N^{(p)}$ over 
$\CM_{N,\,p}$ by the formula 
\begin{equation}
 \label{Psi}
 \tu\big(\Psi_N^{(p)}(P)\big) = \prod_{Q\,\in\,\CG_N^{(p)}}u(P-Q) 
\end{equation}
where $u$ is a coordinate on $\CC_N$ at the identity $O$ and $\tu$ is the 
coordinate on $\CC_N^{(p)}$ induced by $u$ (see \cite[Section 4.3]{Ando00}).  

This isogeny $\Psi_N^{(p)}$ is a {\em deformation of Frobenius} in the sense 
that its restriction over the supersingular point is the Frobenius isogeny 
$\Frob\co C_0 \to C_0^{\,(p)}$, as the $p$-torsion subgroup scheme $C_0[p] = 
0$.  

\begin{defn}
 \label{def:im}
 Let $\p$ be the total power operation on $E$ in \eqref{total}.  With 
 notation as above, a {\em $\CP_N$-model for $\p$} is the data of \ref{m1}--5 
 together with the following.  
 \begin{enumerate}[label=Mod.\,\arabic*, leftmargin=1.8cm, start=6]
  \item \label{m6} a universal deformation 
  $\Psi_N^{(p)}\co \CC_N \to \CC_N^{(p)}$ of $\Frob\co C_0 \to C_0^{\,(p)}$ 
  over $\CM_{N,\,p}$ 
  
  \item \label{m7} an isomorphism over $E^0$ between 
  $\Spf\big(E^0(B\Sigma_p)/I\big)$ and the formal completion of $\CM_{N,\,p}$ 
  at the supersingular point corresponding to $C_0$ 
 \end{enumerate}
\end{defn}

\begin{rmk}
 \label{rmk:STStr}
 The existence of the isomorphism in \ref{m7} follows from Strickland's 
 theorem on Morava E-theories of symmetric groups in terms of rings of 
 functions on the formal moduli \cite[Theorem 1.1]{Str98}, combined with the 
 Serre-Tate theorem.  
\end{rmk}

The formal schemes in \ref{m7} are each of relative dimension $1$ over 
$\Spf\big(W(\cF_p)\big)$ and finite flat of degree $p+1$ over $\Spf(E^0)$.  
We call them {\em Lubin-Tate curves of level $\G_0(p)$} (again, the auxiliary 
$N$ is omitted).

\subsection{Atkin-Lehner involution on the global moduli}
\label{subsec:AL}

Instead of a universal deformation of Frobenius, an alternative viewpoint for 
modeling the total $p$-power operation is through an Atkin-Lehner involution.  

Let us consider a $\CP_N$-model for $\p$ as in Definition \ref{def:im}.  With 
notation from last subsection, there is an automorphism on $\CM_{N,\,p}$, 
induced by the map 
\[
 \left(\CC_N,\,P_0,\,du,\,\CG_N^{(p)}\right) \mapsto 
 \left(\CC_N/\CG_N^{(p)},\,\Psi_N^{(p)}(P_0),\,d\tu,\, 
 \CC_N[p]/\CG_N^{(p)}\right) 
\]
of simultaneous level-$\!\big(\CP_N,\G_0(p)\big)$ structures on (distinct) 
elliptic curves, where $\CC_N[p]$ denotes the $p$-torsion subgroup scheme of 
$\CC_N$ (cf.~\cite[11.2 and 11.3.1]{KM}).  We call this automorphism an {\em 
Atkin-Lehner involution} in connection with Atkin and Lehner's theory of 
modular forms on $\G_0(p)$ \cite[Lemmas 7--10]{AL}.  It is an involution 
since $\CC_N[p] \subset \CC_N$ is of degree $p^2$.  

Given $\Psi_N^{(p)}$ on $\CC_N$, the Atkin-Lehner involution of $\CM_{N,\,p}$ 
induces an isogeny on the quotient curve $\CC_N^{(p)} = \CC_N/\CG_N^{(p)}$, 
namely, 
\[
 \widetilde{\Psi}_N^{(p)}\co \CC_N^{(p)} \to 
 \CC_N^{(p)}/\widetilde{\CG}_N^{(p)} 
\]
where $\widetilde{\CG}_N^{(p)} = \CC_N[p]/\CG_N^{(p)}$.  

As the $p$-divisible group of $\CC_N$ sits in a short exact sequence of its 
formal (connected) and \'etale components \cite[(4) in Section 2.2]{pdiv}, 
the Atkin-Lehner involution interchanges these two sorts of formal and 
\'etale degree-$p$ subgroups.  

Via the correspondences in \ref{m4} and \ref{m7}, given any $x \in E^0$, 
$\p(x)$ is the image $\widetilde{x}$ of $x$ under the Atkin-Lehner 
involution.  
\begin{rmk}
 Note that since $\p$ acts on the scalars $W(\cF_p) \subset E^0$ as the 
 $p$-power Frobenius automorphism, it is {\em not} an involution for this 
 subring.  
\end{rmk}

\begin{conv}
 Henceforth whenever we write a tilde $\sim$ over a symbol, we mean the 
 analogue of this symbol under an Atkin-Lehner involution.  
\end{conv}

\subsection{$h$ and $\A$ as deformation and norm parameters in the local 
moduli}
\label{subsec:parameters}

In \cite[7.7--7.8]{KM}, various pairs of parameters are proposed for the 
local ring of $[\G_0(p)]$ at a supersingular point, when one views the moduli 
problem as an ``open arithmetic surface'' (cf.~[\ib, page xiii and 5.1.1]).  
More explicitly, a presentation for this local ring over a perfect field of 
characteristic $p$ is given in [\ib, 13.4.7], which Rezk applied to produce a 
mod-$p$ presentation for the ring of power operations on $E$ in 
\cite[4.8]{mc1}.  

We shall give an {\em integral} presentation for the local ring of 
$[\G_0(p)]$ in terms of the parameters $T$ and ${\bf N}\big(X(P)\big)$ 
from \cite[7.7]{KM}.  

The parameter $T$ is a uniformizer for the Lubin-Tate ring which carries the 
universal formal deformation (cf.~[\ib, 5.2 (Reg.\,4)]).  We therefore call 
it a {\em deformation parameter} for the local moduli of $[\G_0(p)]$.  Via 
\ref{m4}, it corresponds to $\mu_1 \in E^0$.  

Globally over the moduli, the {\em Hasse invariant at $p$} is a modular form 
over $\BF_p$ of level $1$ and weight $p - 1$ (see [\ib, 12.4] and 
\cite[V\,\,\S4]{AEC}).  Up to $p$-torsion, it lifts to an integral modular 
form on $\CC_N / \CM_N$ for all $p$ and all $N\geq3$ prime to $p$ (see 
\cite[Theorem 1.8.1]{Calegari}, \cite[journal p.\,35]{Buzzard}, and 
\cite{Meier}).\footnote{See also \cite[Example 2.6]{ho} for an explicit 
calculation of the Hasse invariant when $p = 5$ and $N = 4$ as well as an 
illustration of the global-local moduli.  We will return to this below in 
Example \ref{ex:p5N4}.  }  

The Hasse invariant vanishes precisely over the mod-$p$ supersingular locus.  
Its restriction to a formal neighborhood of a supersingular point equals a 
deformation parameter $T$ \cite[12.4.4]{KM}.  

\begin{conv}
 For the reason above, henceforth we will write $h$ (the initial letter of 
 ``Hasse'') for a deformation parameter and also for its corresponding 
 element in $E^0 \subset E^0(B\Sigma_p)/I$.   
\end{conv}

The other parameter ${\bf N}\big(X(P)\big)$ is constructed as a norm 
[\ib, 7.5.2], where $X$ is a coordinate on the formal group of a universal 
elliptic curve, and $P$ a point on the curve of exact order $p$ (see 
[\ib, 5.4]).  We call ${\bf N}\big(X(P)\big)$ a {\em norm parameter} for the 
local moduli of $[\G_0(p)]$.  Via \ref{m7}, it should correspond to an 
element in $E^0(B\Sigma_p)/I$ whose powers generate this ring as a free 
module over $E^0$ of rank $p+1$ (the degree of $\CM_{N,\,p}$ over $\CM_N$) by 
Weierstrass preparation.  

Observe that from the definition \eqref{Psi} of the universal degree-$p$ 
isogeny $\Psi_N^{(p)}$ over $\CM_{N,\,p}$, setting 
\begin{equation}
 \label{A}
 \A \ce \prod_{Q\,\in\,\CG_N^{(p)}-\{O\}}u(Q) 
\end{equation}
we obtain a norm parameter for $[\G_0(p)]$ associated with the coordinate 
$u$.  This {\em $\G_0(p)$-norm} is a modular form on $\CC_N/\CM_{N,\,p}$.  We 
will also write $\A$ for the aforementioned element in $E^0(B\Sigma_p)/I$ 
which corresponds to the restriction of this modular form over the formal 
neighborhood of the supersingular point.\footnote{This is a footnote for the 
expert.  In \cite{p3, ho}, we chose distinct symbols $\K$ for the modular 
form constructed as a $\G_0(p)$-norm, and $\A$ for the modular function as a 
multiple of $\K$ by modular unit of weight $p$, passing from a weighted 
projective space to one of its affine local charts (cf.~last footnote).  As 
this perspective is not involved with addressing the central problem of the 
current paper, to ease notation and make the exposition more accessible, here 
we have suppressed the difference of symbols.  This global-to-local procedure 
underlies \ref{m4}, \ref{m5}, and \ref{m7} as well as the choice of a 
nonvanishing $1$-form in the moduli problem $\CP_N$ (cf.~\cite[8.1.7.1]{KM}, 
we work locally with $\BG_m\backslash\CP_N = 
\BG_m\!\left\backslash\big(\G_1(N)\times[\o]\big)\right. = \G_1(N)$ auxiliary 
to $[\G_0(p)]$).  }  

Thus $E^0 \cong W(\cF_p)\lb h\rb$ and by Weierstrass preparation there exists 
a unique monic polynomial $w(h,\A)$ in $\A$ of degree $p+1$ with coefficients 
in $E^0$ such that 
\begin{equation}
 \label{ha}
 E^0(B\Sigma_p)/I \cong E^0[\A]\,\big/\big(w(h,\A)\big) = 
 W(\cF_p)\lb h,\A\rb\,\big/\big(w(h,\A)\big) 
\end{equation}

\begin{rmk}
 \label{rmk:cot}
 We observe that this norm parameter $\A$ is the multiple which defines the 
 relative cotangent map at the identity along $\Psi_N^{(p)}$, \ie, 
 $\big(\Psi_N^{(p)}\big)^{\!*}d\tu = \A\cdot du$.  
\end{rmk}

\subsection{The norm parameters $\A$ and $\tA$ near the cusps of $X_0(p)$}
\label{subsec:cusps}

Given the geometric interpretation for the norm parameter $\A$ from Remark 
\ref{rmk:cot}, we next consider the compactified moduli 
$\CMB_{N,\,p}$ over $\CMB_N$ and determine the values of $\A$, as a modular 
form, at the cusps of $\CMB_{N,\,p}$ (cf.~\cite[1.13 and 1.11]{padicprop}).  

Recall that $\CMB_{N,\,p}-\CM_{N,\,p}$ is finite \'etale over $\BZ[1/N]$.  
Over $\BZ[1/N, \zeta_N]$ with $\zeta_N$ a primitive $N$'th root of unity, it 
is a disjoint union of sections, called the cusps of $\CMB_{N,\,p}$, two of 
which lie over each cusp of $\CMB_N$.  Among each pair of the two cusps, one 
is \'etale over $\CMB_N$, which corresponds to the (\'etale) subgroups $H_i$ 
of the Tate curve $\Tate(q^N)$ generated by $(\zeta_p^i \, q^{1/p})^N$, with 
$\zeta_p$ a primitive $p$'th root of unity and $i = 0$, $1$, \ldots, $p-1$.  
The other cusp is ramified over $\CMB_N$, corresponding to the (formal) 
subgroup $\mmu_p$ of $\Tate(q^N)$ generated by $\zeta_p$ so that the quotient 
$\Tate(q^N)/\mmu_p = \Tate(q^{N\,\!p})$.  These appear in the literature as 
the unramified and ramified cusps of $X_0(p)$, respectively, without 
reference to the auxiliary level-$N$ structure.  

Recall from Section \ref{subsec:AL} that the Atkin-Lehner involution on 
$\CM_{N,\,p}$ interchanges the two sorts of \'etale and formal degree-$p$ 
subgroups of the $p$-divisible group.  It thus extends to an automorphism on 
$\CMB_{N,\,p}$ which interchanges the two sorts of unramified and ramified 
cusps, respectively.  

In [\ib, 1.11], to discuss Hecke operators, Katz gave a detailed analysis of 
the degree-$p$ isogenies $\pi$ from $\Tate(q^N)$, each with one of the above 
subgroups as kernel.  They are defined over the punctured formal neighborhood 
$\BZ[1/N,\zeta_N][1/p,\zeta_p]\lp q^{1/p}\rp$ around the cusps.  In 
particular, Katz calculated the cotangent maps to their dual isogenies 
$\widecheck\pi$ and obtained $\widecheck\pi^*(\o_\can) = \o_\can$ near the 
ramified cusps whereas $\widecheck\pi^*(\o_\can) = p\cdot\o_\can$ near the 
unramified cusps [\ib, lines 4--5 on book p.\,91].  Here, $\o_\can$ is the 
canonical $1$-form on a Tate curve (cf.~\cite[T.2 in 8.8]{KM}).  

Recall from last subsection that the norm parameter $\A$ is the multiple 
which defines the cotangent map to the universal degree-$p$ isogeny 
$\Psi_N^{(p)}$ over $\CM_{N,\,p}$.  We now determine the values of $\A$, as a 
modular form, at the cusps from the above calculation of Katz.  For this, we 
need to carefully analyze the dual isogenies and $1$-forms involved, by 
imposing conditions on the supersingular elliptic curve $C_0$ and the 
coordinate $u$ on $\CC_N$ of a $\CP_N$-model as follows (cf.~\cite[Section 
3.3]{ho}).  

Let us introduce the following strengthened version of \ref{m1}.  
\begin{enumerate}[label=Mod.\,1$^\text{+}$, leftmargin=2.1cm]
 \item \label{m1+} a supersingular elliptic curve $C_0$ over $\BF_{p^2}$ 
 whose $p^2$-power Frobenius endomorphism equals the map of multiplication by 
 $(-1)^{\,p-1}p$, that is, $\Frob^2 = (-1)^{\,p-1}[p]$ 
\end{enumerate}

\begin{rmk}
 \label{rmk:Poonen}
 By \cite{Poonen}, given any supersingular elliptic curve $C/\cF_p$, there 
 exists such a $C_0$ isomorphic to $C$ over $\cF_p$ (cf.~\cite[Lemma 
 3.21]{Poonenetal}, \cite[Remark 3.3]{p3}, and \cite[3.8]{mc1}).  In 
 particular, the dual isogeny of $\Frob\co C_0 \to C_0^{\,(p)}$ is the 
 Frobenius isogeny out of $C_0^{\,(p)}$ for all $p \neq 2$.  
\end{rmk}

\begin{lem}
 \label{lem:cusps}
 Let $C_0$ be in \ref{m1+} and $\CC_N$ be in \ref{m2}.  Then there exists a 
 coordinate $u^\text{+}$ on $\CC_N$ such that its associated modular form 
 $\A$ in \eqref{A} equals $p$ at the ramified cusps and $(-1)^{\,p-1}$ at the 
 unramified cusps of $\CMB_{N,\,p}$.  The definition of $u^\text{+}$ may 
 require an extension of scalars on $\CM_{N,\,p}$.  
\end{lem}
\begin{proof}
 Let $[X,Y,Z]$ be the homogeneous Weierstrass coordinates of $\CC_N$, with 
 the identity section $O = [0,1,0]$.  Let $u = X/Y$ so that $u(O) = 0$.  Then 
 locally $u$ is a coordinate on the formal group $\HCC_N$ 
 (cf.~\cite[IV\,\,\S1]{AEC}).  
 
 Let $\Psi_N^{(p)}\co \CC_N \to \CC_N^{(p)}$ over $\CM_{N,\,p}$ be the 
 universal degree-$p$ isogeny from \eqref{Psi} constructed using the above 
 coordinate $u$.  Recall from earlier in this subsection Katz's calculation 
 of degree-$p$ isogenies on $\Tate(q^N)$ over 
 $\BZ[1/N,\zeta_N][1/p,\zeta_p]\lp q^{1/p}\rp$.  Up to a unique isomorphism 
 $\sigma_p$, the universal isogeny $\Psi_N^{(p)}$ restricts over a punctured 
 disc around a ramified cusp as 
 \[
  \pi_p\co \Tate(q^N) \to \Tate(q^N)/\mmu_p = \Tate(q^{N\,\!p}), \qquad~~ 
  q\mapsto q^{\,p}~~~ 
 \]
 with $\widecheck\pi_p^*(\o_\can) = \o_\can$ and thus $\pi_p^*(\o_\can) = 
 p\cdot\o_\can$.  Around an unramified cusp $\Psi_N^{(p)}$ restricts up to a 
 unique isomorphism $\sigma_0$ as 
 \[
  \pi_0\co \Tate(q^N) \to \Tate(q^N)/H_0 = \Tate(q^{N/p}), \qquad q\mapsto 
  q^{1/p} 
 \]
 with $\widecheck\pi_0^*(\o_\can) = p\cdot\o_\can$ and thus $\pi_0^*(\o_\can) 
 = \o_\can$.  
 
 Let $\lambda$ be the unit in $\BZ[1/N,\zeta_N]$ such that $du$ restricts to 
 $\lambda\cdot\o_\can$ on $\Tate(q^N)$.  Let $\nu_p$ and $\nu_0$ be the units 
 in $\BZ[1/N,\zeta_N][1/p,\zeta_p]$ such that $d\tu$ restricts to 
 $\nu_p\cdot\lambda^p\cdot\o_\can$ on $\Tate(q^{N\,\!p})$ and to 
 $\nu_0\cdot\lambda^p\cdot\o_\can$ on $\Tate(q^{N/p})$.  The latter two units 
 arise from the isomorphisms $\sigma_p$ and $\sigma_0$ respectively.  Since 
 $\big(\Psi_N^{(p)}\big)^{\!*}d\tu = \A\cdot du$, comparing this identity to 
 those in last paragraph, we see that $\A = \nu_p\lambda^{p-1}p$ at a 
 ramified cusp and $\A = \nu_0\lambda^{p-1}$ at an unramified cusp.  
 
 On the other hand, over the chosen supersingular point in the $\CP_N$-model, 
 by construction $\Psi_N^{(p)}$ restricts to $\Frob\co C_0 \to C_0^{\,(p)}$.  
 Thus by rigidity \cite[2.4.2]{KM} the identity $\Frob^2 = (-1)^{\,p-1}[p]$ 
 from \ref{m1+} lifts to 
 \begin{equation}
  \label{pp}
  \widetilde{\Psi}_N^{(p)}\circ\Psi_N^{(p)} = \tau\circ(-1)^{\,p-1}[p] 
 \end{equation}
 where $\widetilde{\Psi}_N^{(p)}\co \CC_N^{(p)} \to 
 \CC_N^{(p)}/\widetilde{\CG}_N^{(p)}$ is the Atkin-Lehner involution of 
 $\Psi_N^{(p)}$ from Section \ref{subsec:AL}, with $\widetilde{\CG}_N^{(p)} = 
 \CC_N[p]/\CG_N^{(p)}$, and $\T\co \CC_N/\CC_N[p] \to 
 \CC_N^{(p)}/\widetilde{\CG}_N^{(p)}$ is the canonical isomorphism inducing 
 the identity map on relative cotangent spaces.  Comparing \eqref{pp} to 
 $\widetilde{\pi}_p\circ\pi_p = \widetilde{\pi}_0\circ\pi_0 = [p]$ around the 
 cusps \cite[last line on book p.~90 and lines 1--3 on p.~91]{padicprop}, we 
 see that $v_0\cdot v_p = (-1)^{\,p-1}$ as $\widetilde{v}_p = v_0$.  
 
 Set $u^\text{+} \ce \nu_p^{-1/(p-1)}\lambda^{-1}\cdot u = 
 -\nu_0^{-1/(p-1)}\lambda^{-1}\cdot u$.  It is straightforward to check that 
 the $\G_0(p)$-norm $\A$ associated to this coordinate $u^\text{+}$ takes the 
 desired values at the cusps.  
\end{proof}

\begin{ex}[{\cite[Section 3]{h2p2}}]
 \label{ex:p2N3}
 Let $p = 2$ and $N = 3$.  Rezk worked with a $\CP_3$-model where $C_0\co 
 y^2+y = x^3$ in \ref{m1+}, corresponding to the unique mod-$2$ supersingular 
 point in the moduli.  The universal curve $\CC_3\co y^2+axy+y = x^3$ in 
 \ref{m2} has a chosen $3$-torsion point $(0,0)$ (cf.~\cite[Proposition 
 3.2]{tmf3}).  He then set $u = x/y$ in \ref{m3}, which has the property in 
 Lemma \ref{lem:cusps}.  Indeed, his deformation parameter $a$ and norm 
 parameter $d$ satisfy $d^3-ad-2 = 0$.  This factors into $(d-2)(d+1)^2 = 0$ 
 if $a = 3$, which is an integral lift for the Hasse invariant $1$ of the 
 Tate curve.  
\end{ex}

\begin{ex}[{\cite[Sections 2.1--3.1]{p3}}]
 \label{ex:p3N4}
 Let $p = 3$ and $N = 4$.  We worked with a $\CP_4$-model where $C_0\co 
 y^2+xy-y = x^3-x^2$ in \ref{m1} ($\Frob^2 = [-3]$), corresponding to the 
 unique mod-$3$ supersingular point in the moduli.  The universal curve 
 $\CC_4\co y^2+axy+aby = x^3+bx^2$ in \ref{m2} has a chosen $4$-torsion point 
 $(0,0)$.  We set $u = x/y$ in \ref{m3}.  Our deformation parameter $h = 
 a^2+b$ and norm parameter $\A$ satisfy $\A^4-6\A^2+(h-9)\A-3 = 0$, which 
 factors into $(\A-3)(\A+1)^3 = 0$ if $h = 1$.  
\end{ex}

\begin{ex}[{\cite[Sections 2.1--3.1]{ho}}]
 \label{ex:p5N4}
 Let $p = 5$ and $N = 4$.  We worked with a $\CP_4$-model where $\CC_4$ is 
 the same curve as in Example \ref{ex:p3N4}.  Its Hasse invariant at the 
 prime $5$ factors as 
 \[
  a^4-a^2b+b^2 = \big(a^2+2(1+\eta)b\big)\big(a^2+2(1-\eta)b\big) 
 \]
 over $\cF_5$ with $\eta^2 = 2$.  Thus the mod-$5$ supersingular locus 
 consists of two closed points.  We chose $C_0$ in \ref{m1+} which 
 corresponds to the first factor, and again $u = x/y$ in \ref{m3}.  
 
 Setting $h = a^4-16a^2b+26b^2$ as an integral lift of the Hasse invariant 
 above, we were only able to calculate that 
 \begin{equation}
  \label{p5N4}
  \A^6-10\A^5+35\A^4-60\A^3+55\A^2-h\A+5 = 0 
 \end{equation}
 where $\A$ is the (global) modular form constructed as a norm with the 
 coordinate $u$.  We were unable to deduce an equation for a lift of the 
 factor $a^2+2(1+\eta)b$ of the Hasse invariant.  
 
 If we further set $h = 26\equiv1\md5$, \eqref{p5N4} then factors into 
 $(\A-5)(\A-1)^5 = 0$.  Thus $u$ satisfies the property in Lemma 
 \ref{lem:cusps}.  The relation \eqref{p5N4} should specialize to one for the 
 corresponding deformation and norm parameters at the chosen supersingular 
 point.  
\end{ex}

\begin{ex}[{\cite{RezkCorr15}}]
 \label{ex:p5N3}
 Let $p = 5$ and $N = 3$.  We worked with a $\CP_3$-model where $\CC_3$ is 
 the same curve as in Example \ref{ex:p2N3}.  Its Hasse invariant at the 
 prime $5$ factors as 
 \[
  -a^4-a = -a(a+1)(a^2-a+1) 
 \]
 Thus the mod-$5$ supersingular locus consists of three closed points.  We 
 chose $C_0$ in \ref{m1} which corresponds to the first factor ($\Frob^2 = 
 [-5]$) and $u = x/y$ in \ref{m3}.  Setting $h = -a^4+19a$ as an integral 
 lift of the Hasse invariant above, we calculated that 
 \begin{equation}
  \label{p5N3}
  \A^6-5a\A^4+40\A^3-5a^2\A^2-h\A-5 = 0 
 \end{equation}
 where $\A$ is the (global) modular form constructed as a norm with the 
 coordinate $u$.  This factors into $(\A+5)(\A-1)^5 = 0$ if we set $a = 3$ so 
 that $h = -24\equiv1\md5$.  Again, the relation \eqref{p5N3} should 
 specialize to one for the corresponding deformation and norm parameters at 
 the chosen supersingular point, which is equivalent to the relation from 
 Example \ref{ex:p5N4} as an equation for a Lubin-Tate curve of level 
 $\G_0(5)$.  
\end{ex}

\begin{rmk}
 \label{rmk:Ando}
 The coordinate $u^\text{+}$ in Lemma \ref{lem:cusps} (and $u$ in Example 
 \ref{ex:p3N4}) should restrict to a distinguished coordinate on the formal 
 group $\HCC_N$ as well as to one on the formal group of $\Tate(q^N)$, which 
 was originally studied by Ando \cite[Theorem 2.5.7]{Ando95}.  We have given 
 an exposition of this fact in \cite[Section 3.3]{ho}, with more details and 
 greater generality in \cite{nc}.  Our results in the current paper are 
 independent of the existence of Ando's coordinates.  
\end{rmk}

\begin{defn}[{\cite[Definition 3.29]{ho}}]
 \label{def:pm}
 Let $\p$ be the total power operation on $E$ in \eqref{total}.  Continuing 
 with Definition \ref{def:im}, we call the following set of data a {\em 
 preferred model for $\p$}.  
 \begin{enumerate}[label=Mod.\,\arabic*$^\text{+}$, leftmargin=2.1cm, start=3]
  \item [Mod.\,0$^\text{\textcolor{white}{+}}$] an integer $N\geq3$ which is 
  prime to $p$ 
  
  \item [Mod.\,1$^\text{+}$] a supersingular elliptic curve $C_0$ over 
  $\BF_{p^2}$ whose $p^2$-power Frobenius endomorphism equals the map of 
  multiplication by $(-1)^{\,p-1}p$, that is, $\Frob^2 = (-1)^{\,p-1}[p]$ 
  
  \item [Mod.\,2$^\text{\textcolor{white}{+}}$] a universal deformation 
  $\CC_N$ of $C_0$ over $\CM_N$ 
  
  \item \label{m3+} a coordinate $u$ on the formal group $\HCC_N$ which 
  extends to a coordinate on $\CC_N$ satisfying the property that the 
  associated modular form of $\G_0(p)$-norm equals $p$ at the ramified cusps 
  and $(-1)^{\,p-1}$ at the unramified cusps of $\CMB_{N,\,p}$.  
  
  \item [Mod.\,4$^\text{\textcolor{white}{+}}$] an isomorphism between 
  $\Spf(E^0)$ and the formal completion of $\CM_N$ at the supersingular point 
  corresponding to $C_0$ 
  
  \item [Mod.\,5$^\text{\textcolor{white}{+}}$] an isomorphism between 
  $\Spf\big(E^0(\BC P^\infty)\big)$ and $\HCC_N$ as formal groups over $E^0$ 
  which sends $x_{_E}\cdot\mu$ to $u$ 

  \item [Mod.\,6$^\text{\textcolor{white}{+}}$] a universal deformation 
  $\Psi_N^{(p)}\co \CC_N \to \CC_N^{(p)}$ of $\Frob\co C_0 \to C_0^{\,(p)}$ 
  over $\CM_{N,\,p}$ 
  
  \item [Mod.\,7$^\text{\textcolor{white}{+}}$] an isomorphism over $E^0$ 
  between $\Spf\big(E^0(B\Sigma_p)/I\big)$ and the formal completion of 
  $\CM_{N,\,p}$ at the supersingular point corresponding to $C_0$ 
 \end{enumerate}
\end{defn}

As discussed above, the existence of the curve in \ref{m1+} follows from 
Remark \ref{rmk:Poonen} and the existence of the coordinate in \ref{m3+} 
follows from Lemma \ref{lem:cusps}.  

Analogous to \cite[Corollary 3.2]{p3}, note that given a preferred model we 
have 
\begin{equation}
 \label{AA}
 \A \cdot \tA = (-1)^{\,p-1}p 
\end{equation}
where $\tA$ is the Atkin-Lehner involution of the norm parameter $\A$, itself 
also a norm parameter (cf.~\cite[row 7 for {$[\G_0(p^n)]$} of table in 
7.7]{KM}).  
 
\begin{rmk}
 \label{rmk:choices}
 Given Definition \ref{def:pm}, let us summarize at this point the dependence 
 on the various choices made therein of the stated formulas in Theorems 
 \ref{thm:me}, \ref{thm:po}, and \ref{thm:DL}.  
 
 Up to isomorphism, the height-$2$ Morava E-theory spectrum $E$ is 
 independent of the choice \ref{m1+} (or \ref{m1}) by Lazard's theorem and 
 the Goerss-Hopkins-Miller theorem, and independent of \ref{m2} by the 
 Lubin-Tate theorem (Remark \ref{rmk:indepc}).  In particular, different 
 choices between \ref{m1} and \ref{m1+} may result in different but 
 equivalent formulas in the theorems (see, \eg, the case $p = 3$ in Remark 
 \ref{rmk:recover}).  
 
 The choice \ref{m2}, \ie, for different values of $N$, does not affect the 
 coefficients in the modular equation \eqref{me} as long as \ref{m1+} (or 
 \ref{m1}) has been fixed (Remark \ref{rmk:indepi}).  Consequently, the 
 formulas in Theorems \ref{thm:po} and \ref{thm:DL} are independent of $N$.  
 
 The choice \ref{m3+} (or \ref{m3}) does affect the coefficients in 
 \eqref{me} as its parameter $\A$ is built explicitly using a coordinate 
 $u$.  Different choices result in different bases in Theorem \ref{thm:po} 
 for the target ring $E^0(B\Sigma_p)/I$ of $\p$ as a free module over $E^0$, 
 and different bases in Theorem \ref{thm:DL} for the Dyer-Lashof algebra $\G$ 
 as an associative ring over $E^0$.  This choice \ref{m3+} can in fact be 
 further strengthened and made unique, though we do not need this here 
 (Remark \ref{rmk:Ando}).  
 
 The rest \ref{m4}--7 of the list are formal, which will not affect the 
 formulas once the choices above have been made.  
 
 As we shall see in the proofs of the theorems below in Sections 
 \ref{sec:pfag} and \ref{sec:pfat}, for a fixed prime $p$, all {\em 
 preferred} models for $\p$ from Definition \ref{def:pm} will give the {\em 
 same} formulas as stated in the theorems.  
\end{rmk}

\section{Proof of Theorem \ref{thm:me}}
\label{sec:pfag}

Choose any preferred model as in Definition \ref{def:pm}.  We may assume that 
the $j$-invariant for the supersingular point of this model lies in $\BF_p$.  
In fact, by a theorem of Deuring (see \cite[Theorem 14.18, combined with 
Proposition 14.15]{Cox}), there exists a supersingular elliptic curve over 
$\BF_p$ for every $p>3$.  This is also true for $p\leq3$ as shown by explicit 
examples \cite[beginning of V\,\,\S4 and Example 4.5]{AEC}.  Thus the 
corresponding $j$-invariant lies in $\BF_p$.  Since the condition on 
$\Frob^2$ in \ref{m1+} involves at most a substitution of supersingular 
curves via an isomorphism over $\cF_p$ (see Remark \ref{rmk:Poonen}), the 
chosen $j$-invariant remains in $\BF_p$.  

Let $j_0 \in \BZ$ be a lift of this supersingular $j$-invariant.  

In the scheme $\CM_{N,\,p}$ representing the simultaneous moduli problem 
$\CP_N \times [\G_0(p)]$, consider a formal neighborhood $U$ which contains 
this single supersingular point.  Note that $U \cong \Spf(A_1)$ by the 
Serre-Tate theorem (see Remark \ref{rmk:STStr}).  

Define a modular function $h \ce j-j_0$, where $j\,(z) = q^{-1}+744+O(q)$ 
with $q = e^{2\pi iz}$ as usual.  Since $j_0$ lifts a supersingular 
$j$-invariant, $h$ then serves as a deformation parameter for $A_0$ and 
$A_1$.  Locally, modulo $p$, it is a restriction of the Hasse invariant 
(cf.~\cite[12.4.4]{KM}).  

Let $\A$ be the norm parameter (locally near the supersingular point) for 
$A_1$ associated with \ref{m3+} of this model.  By Weierstrass preparation 
there exists a unique polynomial 
\begin{equation}
 \label{w}
 w(h,\A) = \A^{\,p+1}+\sum_{i=0}^pw_i\,\A^i 
\end{equation}
with $w_i \in W(\cF_p)\lb h\rb$ such that $A_1 \cong 
A_0[\A]\,\big/\big(w(h,\A)\big)$ (cf.~\eqref{ha}).  

Write $\th$ and $\tA$ for the images of $h$ and $\A$ under the Atkin-Lehner 
involution.  Note that, as a function on the Lubin-Tate curve $\Spf(A_1)$, 
$\tA$ is only locally defined over the moduli scheme $\CM_{N,\,p}$.  
Nevertheless, by construction as a norm parameter, it is the restriction of 
a globally defined modular form (the local uniformizer $u$ in the 
construction can be taken as a fraction $x/y$ between the affine Weierstrass 
coordinates $x$ and $y$, up to a constant unit depending only on the elliptic 
curve $\CC_N$, cf.~the proof of Lemma \ref{lem:cusps}).  This local function 
$\tA$ thus has a $q$-expansion (at the unramified cusp $\infty$), though it 
may differ from the expansion of the global function.  

The following technical lemma is crucial to our proof of the theorem.  

\begin{lem}
 \label{lem:expansion}
 The local function $\tA$ has a $q$-expansion 
 \[
  \tA(z) = \mu_0\,q^{-1}+O(1) = \mu_0(q^{-1}+a_0)+O(q) 
 \]
 for some $\mu_0 \in W(\cF_p)^\times\cap\BZ$ and $a_0 \in \BZ$ such that 
 $a_0\equiv744-j_0\md p$.  
\end{lem}
\begin{proof}
 First, note that $\tA$ has integral coefficients in its $q$-expansion, since 
 $\A$ does by the $q$-expansion principle \cite[Corollary 1.6.2]{padicprop} 
 applied locally (in the sense above).  
 
 Recall from Remark \ref{rmk:cot} that the parameter $\A$ has a geometric 
 interpretation as the multiple in the cotangent map along the $p$-power 
 isogeny $\Psi_N^{(p)}\co \CC_N \to \CC_N^{(p)} = \CC_N/\CG_N^{(p)}$.  Under 
 the Atkin-Lehner involution, the parameter $\tA$ then gives the cotangent 
 map along $\widetilde{\Psi}_N^{(p)}\co \CC_N^{(p)} \to 
 \CC_N^{(p)}/\widetilde{\CG}_N^{(p)}$ where $\widetilde{\CG}_N^{(p)} = 
 \CC_N[p]/\CG_N^{(p)}$.  In particular, by rigidity, we have an identity 
 \[
  \widetilde{\Psi}_N^{(p)}\circ\Psi_N^{(p)} = \T\circ(-1)^{\,p-1}[p] 
 \]
 that lifts $\Frob^2 = (-1)^{\,p-1}[p]$ over the supersingular point, where 
 $\T\co \CC_N/\CC_N[p] \to \CC_N^{(p)}/\widetilde{\CG}_N^{(p)}$ is the 
 canonical isomorphism.  Thus, up to a unit in the global sections of 
 $\CM_{N,\,p}$, $\widetilde{\Psi}_N^{(p)}$ can be identified with the 
 (Verschiebung) isogeny dual to the (Frobenius) isogeny $\Psi_N^{(p)}$ 
 (cf.~the proof of Lemma \ref{lem:cusps}).  
 
 On the other hand, recall the Hasse invariant as defined in 
 \cite[12.4.1]{KM}, modulo $p$, from the tangent map to Verschiebung.  Since 
 {\em locally} the deformation parameter $h = j-j_0$ is an integral lift of 
 the Hasse invariant, we have 
 \begin{equation}
  \label{cong}
  \tA\equiv\mu\,h\md(p,\A) 
 \end{equation}
 for some unit $\mu \in A_0/(p)$.  From this congruence, we deduce the 
 $q$-expansion for $\tA$ as follows.  
 
 Let the lowest exponent of $q$ in the expansion be $m \in \BZ$.  
 
 Suppose $m<-1$.  This leading term of $\tA \in A_1$ cannot come from an 
 element in $h^2\cdot A_0\subset A_1$, because the neighborhood $U$ contains 
 a {\em single} supersingular point and a Hasse invariant has simple zeros 
 (cf.~[\ib, Theorem 12.4.3]).  Thus it must come from an element in $\A\cdot 
 A_1$.  
 
 {\em Case 1.}  Let us suppose that the coefficient of $q^m$ is not divisible 
 by $p$.  Since $\A\cdot\tA = (-1)^{\,p-1}p$ as in \eqref{AA}, the modular 
 function $\A$ must then have a leading coefficient divisible by $p$, which 
 leads to a contradiction (we deduced in last paragraph that the leading term 
 of $\tA$ came from an element in $\A \cdot A_1$).  
 
 {\em Case 2.}  If the coefficient of $q^m$ is divisible by $p$, then 
 $\A\cdot\tA = (-1)^{\,p-1}p$ implies that $\tA\equiv0\md p$, as $\A$ becomes 
 invertible in this case.  This is again a contradiction, since $\tA$ only 
 becomes zero modulo $p$ at precisely the closed supersingular point.  
 
 We have therefore shown that $m\geq-1$.  
 
 Now, by the argument from Case 2, we deduce that the coefficient of $q^m$ is 
 not divisible by $p$ (there is no more contradiction in Case 1, now that 
 $m\geq-1$).  This time, $\A\cdot\tA = (-1)^{\,p-1}p$ implies that 
 $\A\equiv0\md p$,\footnote{More precisely, $\A\equiv0\md p$ near the 
 unramified cusps, not contradicting the stated congruence in Theorem 
 \ref{thm:me}.  } as $\tA$ is invertible.  Thus the congruence \eqref{cong} 
 strengthens to be 
 \[
  \tA\equiv\mu\,h\md p 
 \]
 and the claimed $q$-expansion for $\tA$ follows.  
\end{proof}

We continue with the proof of the theorem.  By Lemma \ref{lem:expansion}, for 
$2\leq i\leq p$, there exist {\em constants} $\tw_i \in p\BZ$ such that 
\begin{equation}
 \label{key}
 \tA^{\,p}+\tw_p\,\tA^{\,p-1}+\cdots+\tw_2\,\tA = \mu_0^{\,p}q^{-p}+O(1) 
\end{equation}
On the other hand, we have 
\begin{equation}
 \label{th}
 \th\,(z) = j\,(pz)-j_0 = q^{-p}+O(1) 
\end{equation}
Comparing the two displays above, we then have 
\[
 \tA^{\,p}+\tw_p\,\tA^{\,p-1}+\cdots+\tw_2\,\tA = \mu_0^{\,p}\,\th+K+O(q) 
\]
for some $K \in \BZ$.  Passing to the mod-$p$ reduction of this identity, we 
see that $K \in p\BZ$.  Therefore, by an abuse of notation, we can instead 
choose a deformation parameter $h$ such that 
\[
 \tA^{\,p}+\tw_p\,\tA^{\,p-1}+\cdots+\tw_2\,\tA = \th+O(q) 
\]
without changing $\tA$ and $\tw_i$ that we already obtained.  From this we 
have 
\begin{equation}
 \label{tAth}
 \tA^{\,p+1}+\tw_p\,\tA^{\,p}+\cdots+\tw_2\,\tA^2 = \th\,\tA+O(1) 
\end{equation}
We claim that the last term $O(1)$ is constant.  In fact, since $A_1$ is a 
free module over $A_0$ of rank $p+1$, $\tA^{\,p+1}$ can be expressed as a 
linear combination of $\tA^i$, $0\leq i\leq p$, with coefficients polynomials 
in $\th$ over $W(\cF_p)$.  Given the $q$-expansions of $\th$ in \eqref{th} 
and of $\tA$ from Lemma \ref{lem:expansion}, we see that the term $O(1)$ must 
be a constant integer, corresponding to the basis element $\tA^i$, $i = 0$, 
with coefficient an integer.  

Thus the terms $\tw_i$ and $O(1)$ in \eqref{tAth} are all constant integers.  

Applying the Atkin-Lehner involution to \eqref{tAth}, we then conclude that 
except for $i = 1$, the coefficients $w_i$ in \eqref{w} are all constant 
integers.  It remains to determine their values.  This follows from Lemma 
\ref{lem:cusps} by holomorphicity of the modular functions $w_i$ over the 
(complex-analytic) moduli scheme, as we move from the formal neighborhood of 
the supersingular point to the cusps via the global functions.  See Figure 
\ref{fig}.  The two pictures illustrate the same process (cf.~\cite[Figure 
2]{Calegari} and \cite[book p.\,290]{DR}) with $X_0(11)$ of genus $1$ as an 
example.  
\begin{figure}
 \centering
 \setcounter{figure}{6}
 \includegraphics[scale=.68]{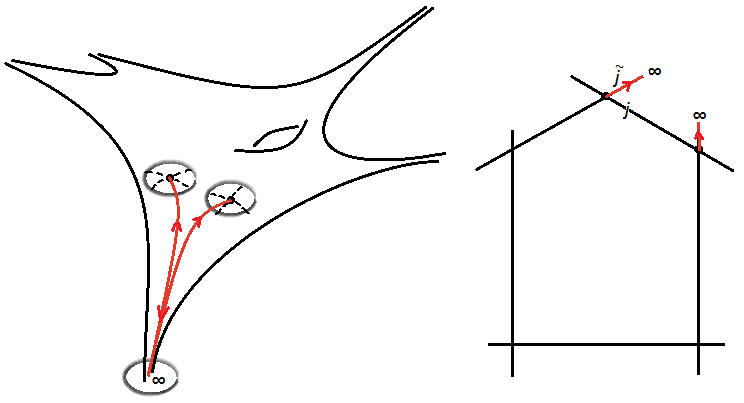}
 \caption{Transporting functions over the moduli}
 \label{fig}
\end{figure}
Explicitly, comparing the former local equation 
\[
 \A^{\,p+1}+w_p\,\A^{\,p}+\cdots+w_2\,\A^2-h\,\A+w_0 = 0 
\]
to the latter local equation 
\[
 (\A-p)\big(\A-(-1)^{\,p-1}\big)^p = 0 
\]
we obtain the polynomial $w(h,\A)$ in Theorem \ref{thm:me} ($h$ becomes 
$(-1)^{\,p^2+1}+p^2\equiv1\md p$, consistent with \cite[12.4.2]{KM}).  This 
completes the proof of Theorem \ref{thm:me}.

\section{Proof of Theorems \ref{thm:po} and \ref{thm:DL}}
\label{sec:pfat}

Theorem \ref{thm:po}\,(i) follows from \cite[Theorem 1.1]{Str98}.  We show 
the remaining parts in two steps.

\subsection{The total power operation formula and the Adem relations}
\label{subsec:Adem}

To compute $\p(h)$, we follow the recipe illustrated in \cite[Example 
3.5]{ho} and generalize [\ib, proof of Proposition 6.4].  By \eqref{me} and 
\eqref{AA}, since 
\[
 \begin{split}
  w(h,\A) = & ~w_{p+1}\A^{\,p+1}+\cdots+w_1\A+w_0 \\
          = & ~w_{p+1}\A^{\,p+1}+\cdots-h\,\A+\tA\,\A 
 \end{split}
\]
is zero in the target ring of $\p$, we have 
\[
 \hskip.5cm h = w_{p+1}\A^{\,p}+\cdots+w_2\,\A+\tA 
\]
where $w_{p+1}$, \ldots, $w_2$ are constants, \ie, they do not involve $h$, 
as computed in Theorem \ref{thm:po}\,(i).  Applying the Atkin-Lehner 
involution to this identity, we then obtain 
\[
 \hskip-.9cm \p(h) = \th = w_{p+1}\,\tA^{\,p}+\cdots+w_2\,\tA+\A 
\]
For $1\leq\T\leq p$, we need only express each $\tA^\T$ as a polynomial in 
$\A$ of degree at most $p$ with coefficients in $E^0$.  The constant terms 
$d_{0,\T}$ of these polynomials have been computed as $d_\T$ in \cite[proof 
of Proposition 6.4]{ho}.  The same method there applies to give the stated 
formulas for the higher coefficients $d_{i,\T}$ with $1\leq i\leq 
p$.\footnote{In fact, those formulas hold for all $0\leq i\leq p$ and $\T\geq 
1$ in expressing $\tA^\T$ from $\A^i$ (with the convention that $w_\T = 0$ if 
$\T>p+1$).  }  This completes the proof of Theorem \ref{thm:po}.  

To derive the Adem relations, we generalize \cite[proof of Proposition 
3.6\,(iv)]{p3} (cf.~\cite[proof of Proposition 6.4]{ho}).  In view of the 
relation $\A\cdot\tA = w_0$, we have 
\[
 \begin{split}
  \p\big(\p(x)\big) = & ~\p\sum_{j=0}^pQ_j(x)\,\A^{\,j} \\
                    = & ~\sum_{j=0}^p\p\big(Q_j(x)\big)\,\p(\A)^{\,j} \\
                    = & ~\sum_{j=0}^p\Bigg(\sum_{i=0}^pQ_iQ_j(x)\,\A^i\Bigg) 
                        \tA^{\,j} \\
                    = & ~\sum_{j=0}^p\Bigg(\sum_{i=0}^jw_0^i\,Q_iQ_j(x)\, 
                        \tA^{\,j-i}+\sum_{i=j+1}^pw_0^{\,j}\,Q_iQ_j(x)\, 
                        \A^{i-j}\Bigg) \\
                    = & ~\sum_{k=0}^p\A^k\Bigg(\sum_{j=0}^p\sum_{i=0}^jw_0^i 
                        \,d_{k,\,j-i}\,Q_iQ_j(x)+\sum_{j=0}^{p-k}w_0^{\,j}\, 
                        Q_{k+j}Q_j(x)\Bigg) 
 \end{split}
\]
where $d_{k,0} = 0$ if $k>0$ (and $d_{0,0} = 1$ from before).  Write the 
expression in last line above as $\sum_{k=0}^p\Psi_k(x)\,\A^k$.  For $1\leq 
k\leq p$, the vanishing of each $\Psi_k$ then gives the stated relation for 
$Q_k Q_0$.

\subsection{The commutation relations}
\label{subsec:comm}

To facilitate computations, let us make a change of variables $\B \ce 
\A+(-1)^{\,p}$.  We then have 
\[
 \begin{split}
  \p(hx) = & ~\p(h)\,\p(x) \\
         = & ~\sum_{i=0}^pQ_i(h)\,\A^i\,\sum_{j=0}^pQ_j(x)\,\A^{\,j} \\
         = & ~\sum_{m=0}^{2p}\!\Bigg(\!\!\sum_{\stackrel{\scriptstyle i+j=m} 
             {0\,\leq\,i,\,j\,\leq\,p}}\!\!Q_i(h)\,Q_j(x)\Bigg)\A^m \\
         = & ~\sum_{m=0}^{2p}\!\Bigg(\!\!\sum_{\stackrel{\scriptstyle i+j=m} 
             {0\,\leq\,i,\,j\,\leq\,p}}\!\!Q_i(h)\,Q_j(x)\Bigg)\!\big(\B+ 
             (-1)^{\,p+1}\big)^m \\
         = & ~\sum_{m=0}^{2p}\!\Bigg(\!\!\sum_{\stackrel{\scriptstyle i+j=m} 
             {0\,\leq\,i,\,j\,\leq\,p}}\!\!Q_i(h)\,Q_j(x)\Bigg)\sum_{n=0}^m 
             \!\ch{m}{n}\B^n(-1)^{(p+1)(m-n)} \\
         = & ~\sum_{n=0}^{2p}e_n\,\B^n 
 \end{split}
\]
where 
\[
  e_n = \sum_{m=n}^{2p}(-1)^{(p+1)(m-n)}\ch{m}{n}\!\!\sum_{\stackrel 
  {\scriptstyle i+j=m}{0\,\leq\,i,\,j\,\leq\,p}}\!\!Q_i(h)\,Q_j(x) 
\]
Formulas for the terms $Q_i(h)$ above are given by Theorem 
\ref{thm:po}\,(ii).  Note that $Q_1(h)$ includes a term of 1.  

We now reduce $\p(hx)$ above modulo $w(h,\A)$, by first rewriting the latter 
as a polynomial in $\B$: 
\[
 \begin{split}
  w(h,\A) = & ~(\A-p)\big(\A+(-1)^{\,p}\big)^p-\big(h-p^2+(-1)^{\,p}\big)\A \\
          = & ~\big(\B+(-1)^{\,p+1}-p\big)\B^{\,p}-\big(h-p^2+(-1)^{\,p}\big) 
              \big(\B+(-1)^{\,p+1}\big) \\
          = & ~\B^{\,p+1}+v_p\,\B^{\,p}+v_1\,\B+v_0 
 \end{split}
\]
where $v_p = (-1)^{\,p+1}-p$, $v_1 = -\big(h-p^2+(-1)^{\,p}\big)$, and $v_0 = 
(-1)^{\,p+1}v_1$.  We then carry out a long division of $\p(hx)$ by $w(h,\A)$ 
with respect to $\B$ and obtain 
\[
 \p(hx)\equiv\sum_{j=0}^p\,f_{\,j}\,\B^{\,j}\md w(h,\A) 
\]
where 
\[
 ~~f_{\,j} = \left\{ \hskip-.1cm
 \begin{array}{lll}
  \displaystyle e_p-v_1e_{2p}+v_p\sum_{m=0}^{p-1}(-1)^{m+1}e_{p+1+m}\,v_p^m 
                && \,j=p \\
  \displaystyle e_j+v_0\sum_{m=0}^{p-j-1}(-1)^{m+1}e_{p+j+1+m}\,v_p^m+v_1 
                \sum_{m=0}^{p-j}(-1)^{m+1}e_{p+j+m}\,v_p^m && 0<j<p \\
  \displaystyle e_0+v_0\sum_{m=0}^{p-1}(-1)^{m+1}e_{p+1+m}\,v_p^m && \,j=0 
 \end{array}
 \right. 
\]
Thus we can rewrite 
\[
 \begin{split}
  \p(hx) = & ~\sum_{j=0}^p\,f_{\,j}\,\big(\A+(-1)^{\,p}\big)^{\,j} \\
         = & ~\sum_{j=0}^p\,f_{\,j}\,\sum_{i=0}^j\!\ch{\,j\,}{i}\A^i\, 
             (-1)^{\,p\,(\,j-i)} \\
         = & ~\sum_{i=0}^p\Bigg[\sum_{j=i}^p(-1)^{\,p\,(\,j-i)}\ch{\,j\,}{i} 
             \,f_{\,j}\Bigg]\A^i 
 \end{split}
\]
On the other hand, $\p(hx) = \sum_{i=0}^pQ_i(hx)\,\A^i$.  Comparing this to 
the last expression for $\p(hx)$ above, term by term, we obtain the 
commutation relations as stated.  This completes the proof of Theorem 
\ref{thm:DL}.


\vspace{-.5in}
\begin{thebibliography}

\section*{\leftskip=-.44in References \vspace{.15in}}

\bibitem[Ahlgren2003]{Ahlgren}
Scott Ahlgren, \emph{The theta-operator and the divisors of modular forms on 
  genus zero subgroups}, Math. Res. Lett. \textbf{10} (2003), no.~6, 
  787--798. \MRn{2024734}

\bibitem[Ando1995]{Ando95}
Matthew Ando, \emph{Isogenies of formal group laws and power operations in 
  the cohomology theories {$E\sb n$}}, Duke Math. J. \textbf{79} (1995), 
  no.~2, 423--485. \MRn{1344767}

\bibitem[Ando2000]{Ando00}
\bysame, \emph{Power operations in elliptic cohomology and representations of 
  loop groups}, Trans. Amer. Math. Soc. \textbf{352} (2000), no.~12, 
  5619--5666. \MRn{1637129}

\bibitem[Ando-Hopkins-Strickland2004]{AHS04}
Matthew Ando, Michael~J. Hopkins, and Neil~P. Strickland, \emph{The sigma 
  orientation is an {$H\sb \infty$} map}, Amer. J. Math. \textbf{126} (2004), 
  no.~2, 247--334. \MRn{2045503}

\bibitem[Atkin-Lehner1970]{AL}
A.~O.~L. Atkin and J.~Lehner, \emph{Hecke operators on {$\Gamma \sb{0}(m)$}}, 
  Math. Ann. \textbf{185} (1970), 134--160. \MRn{0268123}

\bibitem[Baker \ea2005]{Poonenetal}
Matthew~H. Baker, Enrique Gonz{\'a}lez-Jim{\'e}nez, Josep Gonz{\'a}lez, and 
  Bjorn Poonen, \emph{Finiteness results for modular curves of genus at least 
  2}, Amer. J. Math. \textbf{127} (2005), no.~6, 1325--1387. \MRn{2183527}

\bibitem[Behrens-Rezk2017]{BKTAQ}
Mark Behrens and Charles Rezk, \emph{The Bousfield-Kuhn functor and 
  topological Andr\'e-Quillen cohomology}, available at 
  \url{http://www3.nd.edu/~mbehren1/papers/BKTAQ8.pdf}.

\bibitem[Bruinier-Kohnen-Ono2004]{BKO}
Jan~H. Bruinier, Winfried Kohnen, and Ken Ono, \emph{The arithmetic of the 
  values of modular functions and the divisors of modular forms}, Compos. 
  Math. \textbf{140} (2004), no.~3, 552--566. \MRn{2041768}

\bibitem[Bruner \ea1986]{H_infty}
R.~R. Bruner, J.~P. May, J.~E. McClure, and M.~Steinberger, \emph{{$H\sb 
  \infty$} ring spectra and their applications}, Lecture Notes in 
  Mathematics, vol. 1176, Springer-Verlag, Berlin, 1986. \MRn{836132}

\bibitem[Buzzard2003]{Buzzard}
Kevin Buzzard, \emph{Analytic continuation of overconvergent eigenforms}, J. 
  Amer. Math. Soc. \textbf{16} (2003), no.~1, 29--55. \MRn{1937198}

\bibitem[Calegari2013]{Calegari}
Frank Calegari, \emph{Congruences between modular forms}, available at 
  \url{http://swc.math.arizona.edu/aws/2013/2013CalegariLectureNotes.pdf}.  

\bibitem[Choi2006]{Choi}
D.~Choi, \emph{On values of a modular form on {$\Gamma\sb 0(N)$}}, Acta 
  Arith. \textbf{121} (2006), no.~4, 299--311. \MRn{2224397}

\bibitem[Cox2013]{Cox}
David~A. Cox, \emph{Primes of the form {$x\sp 2 + ny\sp 2$}}, second ed., 
  Pure and Applied Mathematics (Hoboken), John Wiley \& Sons, Inc., Hoboken, 
  NJ, 2013, Fermat, class field theory, and complex multiplication. 
  \MRn{3236783}

\bibitem[Deligne-Rapoport1973]{DR}
P.~Deligne and M.~Rapoport, \emph{Les sch\'{e}mas de modules de courbes 
  elliptiques}, 143--316. Lecture Notes in Math., Vol. 349. \MRn{0337993}

\bibitem[Devinatz-Hopkins-Smith1988]{DHS}
Ethan~S. Devinatz, Michael~J. Hopkins, and Jeffrey~H. Smith, \emph{Nilpotence 
  and stable homotopy theory. {I}}, Ann. of Math. (2) \textbf{128} (1988), 
  no.~2, 207--241. \MRn{960945}

\bibitem[Goerss-Hopkins2004]{GH}
P.~G. Goerss and M.~J. Hopkins, \emph{Moduli spaces of commutative ring 
  spectra}, Structured ring spectra, London Math. Soc. Lecture Note Ser., 
  vol. 315, Cambridge Univ. Press, Cambridge, 2004, pp.~151--200. 
  \MRn{2125040}

\bibitem[Hopkins1999]{coctalos}
Mike Hopkins, \emph{Complex oriented cohomology theories and the language of 
  stacks}, available at 
  \url{http://web.math.rochester.edu/people/faculty/doug/otherpapers/\name}.  

\bibitem[Hopkins-Smith1998]{HS}
Michael~J. Hopkins and Jeffrey~H. Smith, \emph{Nilpotence and stable homotopy 
  theory. {II}}, Ann. of Math. (2) \textbf{148} (1998), no.~1, 1--49. 
  \MRn{1652975}

\bibitem[Huan2018]{Huan}
Zhen Huan, \emph{Quasi-elliptic cohomology and its power operations}, J. 
  Homotopy Relat. Struct. \textbf{13} (2018), no.~4, 715--767. \MRn{3870771}

\bibitem[Katz1973]{padicprop}
Nicholas~M. Katz, \emph{{$p$}-adic properties of modular schemes and modular 
  forms}, Modular functions of one variable, {III} ({P}roc. {I}nternat. 
  {S}ummer {S}chool, {U}niv. {A}ntwerp, {A}ntwerp, 1972), Springer, Berlin, 
  1973, pp.~69--190. Lecture Notes in Mathematics, Vol. 350. \MRn{0447119}

\bibitem[Katz-Mazur1985]{KM}
Nicholas~M. Katz and Barry Mazur, \emph{Arithmetic moduli of elliptic 
  curves}, Annals of Mathematics Studies, vol. 108, Princeton University 
  Press, Princeton, NJ, 1985. \MRn{772569}

\bibitem[Lazard1955]{Lazard}
Michel Lazard, \emph{Sur les groupes de {L}ie formels \`a un param\`etre}, 
  Bull. Soc. Math. France \textbf{83} (1955), 251--274. \MRn{0073925}

\bibitem[Lubin-Serre-Tate1964]{LST}
J.~Lubin, J.-P. Serre, and J.~Tate, \emph{Elliptic curves and formal groups}, 
  available at \url{http://www.ma.utexas.edu/users/voloch/lst.html}.  

\bibitem[Lubin-Tate1966]{LT}
Jonathan Lubin and John Tate, \emph{Formal moduli for one-parameter formal 
  {L}ie groups}, Bull. Soc. Math. France \textbf{94} (1966), 49--59. 
  \MRn{0238854}

\bibitem[Lurie2009]{survey}
J.~Lurie, \emph{A survey of elliptic cohomology}, Algebraic topology, Abel 
  Symp., vol.~4, Springer, Berlin, 2009, pp.~219--277. \MRn{2597740}

\bibitem[Mahowald-Rezk2009]{tmf3}
Mark Mahowald and Charles Rezk, \emph{Topological modular forms of level 3}, 
  Pure Appl. Math. Q. \textbf{5} (2009), no.~2, Special Issue: In honor of 
  Friedrich Hirzebruch. Part 1, 853--872. \MRn{2508904}

\bibitem[Meier2016]{Meier}
Lennart Meier, \emph{Lifting the Hasse invariant mod 2}, MathOverflow, 
  \url{http://mathoverflow.net/q/228497} (version: 2016-01-16).  

\bibitem[Milne2017]{Milne}
J.S. Milne, \emph{Modular functions and modular forms}, available at 
  \url{http://www.jmilne.org/math/CourseNotes/MF.pdf}.  

\bibitem[Poonen2010]{Poonen}
Bjorn Poonen, \emph{Supersingular elliptic curves and their ``functorial'' 
  structure over $\BF_{p^2}$}, MathOverflow, 
  \url{http://mathoverflow.net/a/19013} (version: 2010-03-22).  

\bibitem[Quillen1969]{Quillen}
Daniel Quillen, \emph{On the formal group laws of unoriented and complex 
  cobordism theory}, Bull. Amer. Math. Soc. \textbf{75} (1969), 1293--1298. 
  \MRn{0253350}

\bibitem[Ravenel1992]{orange}
Douglas~C. Ravenel, \emph{Nilpotence and periodicity in stable homotopy 
  theory}, Annals of Mathematics Studies, vol. 128, Princeton University 
  Press, Princeton, NJ, 1992, Appendix C by Jeff Smith. \MRn{1192553}

\bibitem[Rezk2008]{h2p2}
Charles Rezk, \emph{Power operations for {M}orava {$E$}-theory of height 2 at 
  the prime 2}. \AX{0812.1320}

\bibitem[Rezk2009]{cong}
\bysame, \emph{The congruence criterion for power operations in {M}orava 
  {$E$}-theory}, Homology, Homotopy Appl. \textbf{11} (2009), no.~2, 
  327--379. \MRn{2591924}

\bibitem[Rezk2012]{mc1}
\bysame, \emph{Modular isogeny complexes}, Algebr. Geom. Topol. \textbf{12} 
  (2012), no.~3, 1373--1403. \MRn{2966690}

\bibitem[Rezk2013a]{h2}
\bysame, \emph{Power operations in Morava $E$-theory: structure and 
  calculations (Draft)}, available at 
  \url{http://www.math.uiuc.edu/~rezk/power-ops-ht-2.pdf}.

\bibitem[Rezk2013b]{RezkCorr13}
\bysame, correspondence, 2013.  

\bibitem[Rezk2015]{RezkCorr15}
\bysame, correspondence, 2015.  

\bibitem[Rezk2017]{Koszul}
\bysame, \emph{Rings of power operations for Morava E-theories are Koszul}. 
  \AX{1204.4831}

\bibitem[Schwede2018]{global}
Stefan Schwede, \emph{Global homotopy theory}, New Mathematical Monographs, 
  vol.~34, Cambridge University Press, Cambridge, 2018. \MRn{3838307}

\bibitem[Silverman2009]{AEC}
Joseph~H. Silverman, \emph{The arithmetic of elliptic curves}, second ed., 
  Graduate Texts in Mathematics, vol. 106, Springer, Dordrecht, 2009. 
  \MRn{2514094}

\bibitem[Strickland1997]{Str97}
Neil~P. Strickland, \emph{Finite subgroups of formal groups}, J. Pure Appl. 
  Algebra \textbf{121} (1997), no.~2, 161--208. \MRn{1473889}

\bibitem[Strickland1998]{Str98}
\bysame, \emph{Morava {$E$}-theory of symmetric groups}, Topology \textbf{37} 
  (1998), no.~4, 757--779. \MRn{1607736}

\bibitem[Tate1967]{pdiv}
J.~T. Tate, \emph{{$p$}-divisible groups}, Proc. {C}onf. {L}ocal {F}ields 
  ({D}riebergen, 1966), Springer, Berlin, 1967, pp.~158--183. \MRn{0231827}

\bibitem[Weinstein2016]{Weinstein}
Jared Weinstein, \emph{Semistable models for modular curves of arbitrary 
  level}, Invent. Math. \textbf{205} (2016), no.~2, 459--526. \MRn{3529120}

\bibitem[Zhu2014]{p3}
Yifei Zhu, \emph{The power operation structure on {M}orava {$E$}-theory of 
  height 2 at the prime 3}, Algebr. Geom. Topol. \textbf{14} (2014), no.~2, 
  953--977. \MRn{3160608}

\bibitem[Zhu2015]{ho}
\bysame, \emph{The Hecke algebra action on Morava E-theory of height 2}, 
  available at \url{https://yifeizhu.github.io/ho.pdf}.  

\bibitem[Zhu2017]{nc}
\bysame, \emph{Norm coherence for descent of level structures on formal 
  deformations}, available at \url{https://yifeizhu.github.io/nc.pdf}.  

\bibitem[Zhu2018]{bkos}
\bysame, \emph{Morava {$E$}-homology of {B}ousfield-{K}uhn functors on 
  odd-dimensional spheres}, Proc. Amer. Math. Soc. \textbf{146} (2018), 
  no.~1, 449--458. \MRn{3723154}
\end{thebibliography}
\renewcommand\refname{}
\newcommand{\AX}[1]{\href{http://arxiv.org/abs/#1}{arXiv:#1}}
\newcommand{\MRn}[1]{\href{http://www.ams.org/mathscinet-getitem?mr=#1}{MR#1}}
\newcommand{\name}{coctalos.pdf}
\wt{.}

\end{document}